\documentclass[leqno,11pt]{amsart}
\usepackage{amssymb}
\usepackage{amsmath}
\usepackage{enumerate}
\usepackage{amsfonts}
\usepackage{tikz}
\usepackage{todonotes}

\newcommand{\dz}{{\text{d}_{G}} }
\newcommand{\supp}{\text{supp }}

\newtheorem{theorem}[equation]{Theorem}
\newtheorem{corollary}[equation]{Corollary}
\newtheorem{lemma}[equation]{Lemma}
\newtheorem{proposition}[equation]{Proposition}

\newtheorem{definition}{Definition}[section]

\numberwithin{equation}{section}

\usepackage{hyperref}
\hypersetup{
    colorlinks=true,
    linkcolor= blue,
    citecolor =cyan,
    urlcolor = teal,
}

\begin{document}

\title[Hardy spaces for the Dunkl harmonic oscillator]{Hardy spaces for the Dunkl harmonic oscillator}

\subjclass[2000]{42B30, 35K08, 42B25, 42B35, 47D03}
\keywords{Rational Dunkl theory, Hardy spaces, Riesz transforms, maximal operators.}

\begin{abstract}
Let $\Delta$ and $L=\Delta -\|\mathbf x\|^2$ be the Dunkl Laplacian and the Dunkl  harmonic oscillator respectively. We define the Hardy space $\mathcal H^1$ associated with the Dunkl harmonic oscillator by means of the nontangential maximal function with respect to the semigroup $e^{tL}$. We prove that the space $\mathcal H^1$  admits characterizations by relevant Riesz transforms and atomic decompositions. The atoms which occur in the atomic decompositions are of local type.
\end{abstract}

\author[Agnieszka Hejna]{Agnieszka Hejna}

\address{Instytut Matematyczny\\
Uniwersytet Wroc\l awski\\
50-384 Wroc\l aw \\
Pl. Grunwal\-dzki 2/4\\
Poland} \email{agnieszka.hejna@math.uni.wroc.pl}

\thanks{
Reseach supported by the National Science Centre, Poland (Narodowe Centrum Nauki), Grant 2017/25/B/ST1/00599.}

\maketitle
\section{Introduction}
The classical real Hardy spaces $H^p$ in $\mathbb R^N$ occurred as boundary values of harmonic functions on $\mathbb{R}_{+}\times \mathbb{R}^{N}$  satisfying generalized Cauchy--Riemann equations together with certain $L^p$ bound conditions (see e.g. Stein--Weiss \cite{SW}). In the seminal paper of Fefferman and Stein~\cite{FS} the spaces $H^p$ were characterized by means of real analysis. So a tempered distribution $f$ belongs to the $H^p(\mathbb R^N)$, $0<p<\infty$, if and only if the maximal function
$\sup_{t>0} |h_t \ast f(\mathbf{x})|$ belongs to $L^p(\mathbb R^N)$, where $h_t$ is the heat kernel of the semigroup $e^{t\Delta}$. Another characterization, which relates the elements of $H^p$ with the system of conjugate harmonic functions, is that by the Riesz transforms. Assume that $(N-1)\slash N<p\leq 1$. Then $f$ is in $H^p$ if and only if the Riesz transforms $\partial_{x_j}(-\Delta)^{-1\slash 2}f$, $j=1,2,...,N$,  are in $L^p(\mathbb R^N)$. An important contribution to the theory is atomic decomposition proved by Coifman \cite{Coifman} for $N=1$  and Latter \cite{Latter} in higher dimensions, which says that every element of $H^p$ can be written as an (infinite) combination of special simple functions called atoms. These characterizations led to generalizations of Hardy spaces on  spaces of homogeneous type or to semigroups of linear operators.

In \cite{conjugate} (see also  \cite{Anker15}, \cite{Dziubanski16},~\cite{Dziubanski18})
the authors develop  a theory of Hardy spaces $H^1$ in the rational Dunkl setting parallel to the classical one. For a root system $R$ in $\mathbb R^N$, a multiplicity function $k \geq 0$, and associated Dunkl Laplacian $\Delta$,  systems of conjugate $(\partial_t^2+\Delta)$-harmonic functions in $\mathbb R_+\times \mathbb R^N$ satisfying a relevant $L^1(dw)$ condition are studied there. Here $w$ denotes a natural measure associated with  $k$ (see~\eqref{eq:measure}). It is proved in \cite{conjugate} that boundary values of such harmonic  functions, which constitute the real Hardy space $H^1_{\Delta}$, can be characterized by relevant Riesz transforms, maximal functions, atomic decompositions, and square functions.

The present paper is a continuation of \cite{conjugate} and concerns real Hardy space $\mathcal{H}^1$ for the Dunkl harmonic oscillator $L =\Delta-\|\mathbf{x}\|^2$. For any $1\leq p<\infty$, an extension of $L $ is the infinitesimal generator the semigroup $K_t=e^{tL}$ on $L^p(w)$, which has the form
\begin{equation*}
K_tf(\mathbf{x})=\int_{\mathbb{R}^N}k_t(\mathbf{x},\mathbf{y})f(\mathbf{y})\,dw(\mathbf{y}),
\end{equation*}
where the integral kernel $k_t(\mathbf{x},\mathbf{y})$ can be expressed by means of generalized Hermite functions (see e.g., \cite{Roesler2}, \cite{Roesler-Voit}).

\begin{definition}\normalfont\label{def:Hardy_Dunkl_Hermite}
 Let $f \in L^{1}(\mathbb{R}^N,dw)$. We say that $f$ belongs to the \textit{Hardy space} $\mathcal{H}^1$ associated with the Dunkl harmonic oscillator $L $ if and only  if
\begin{equation}
f^{*}(\mathbf{x})=\sup_{t>0}|K_tf(\mathbf{x})|
\end{equation}
belongs to $L^1(\mathbb{R}^N,dw)$. The norm in the space is given by
\begin{equation}
\|f\|_{\mathcal{H}^1}=\|f^{*}\|_{L^{1}(\mathbb{R}^N,dw)}.
\end{equation}
\end{definition}
Our goal is to prove that the space $\mathcal H^1$ is equivalently characterized by relevant  Riesz transforms and by atomic decompositions. The atoms for $\mathcal H^1$, which will be elaborately explained in the next section,  are of local type. In order to obtain the characterizations for $\mathcal H^1$ we first define and characterize family local Hardy spaces associated with the Dunkl Laplace operator $\Delta$ (see Section \ref{sec:localHardy}). Further the theory of local the Hardy spaces combined with a detailed analysis of the integral kernels for the semigroups $e^{tL }$ and $e^{t\Delta}$  allow us to obtain our  results.

The theory of Dunkl operators was initiated in the series of works of Dunkl   \cite{Dunkl0}--\cite{Dunkl3} and continued by many mathematicians afterwards (see, e.g. \cite{dJ}, \cite{RoeslerDeJeu}, \cite{Roesler2}, \cite{Roesler4}, \cite{ThangaveluXu}). We refer the reader to lecture notes \cite{Rosler} and \cite{Roesler-Voit}  for more information and references.

In \cite{HMMLY} the authors provided a general approach to the theory of Hardy spaces associated with semigroups satisfying Davies--Gaffney estimates and in particular Gaussian bounds.  We want to  emphasise that the integral kernels associated with the Dunkl Laplace operator and Dunkl harmonic oscillator semigroups do not satisfy Gaussian bounds.   Therefore the methods developed in \cite{HMMLY}  cannot be directly applied.

\section{Preliminaries and statements of the results}

On the Euclidean space $\mathbb R^N$ equipped with the standard inner product $\langle \mathbf{x}, \mathbf{y} \rangle = \sum_{j=1}^{N}x_jy_j$ with the associated norm $\|\mathbf{x}\|$ and a nonzero vector $\alpha\in\mathbb R^N$, the reflection $\sigma_\alpha$ with respect to the orthogonal hyperplane $\alpha^\perp$  is
given by
\begin{align*}
\sigma_\alpha (\mathbf{x})= \mathbf{x}-2\frac{\langle  \mathbf{x},\alpha\rangle}{\| \alpha\| ^2}\alpha.
\end{align*}

A finite set $R\subset \mathbb R^N\setminus\{0\}$ is called a {\it root system} if $\sigma_\alpha (R)=R$ for every $\alpha\in R$. We shall consider any normalized  root systems, that is, $\|\alpha\|^2=2$ for every $\alpha\in R$. The finite group $G\subset O(N)$ generated by the reflections $\sigma_\alpha$ is called the {\it Weyl group} ({\it reflection group}) of the root system. A function $k:R\to\mathbb C$ is called a {\it multiplicity function} if $k$  is  $G$-invariant. In this paper we shall assume that $k\geq 0$.

Given a root system $R$ and a multiplicity function $k$, the \textit{Dunkl operators} $T_\xi$, $\xi \in \mathbb{R}^N$, are the following deformations of directional derivatives $\partial_\xi$ by difference operators\,:
\begin{align*}
T_\xi f(\mathbf{x})
&=\partial_\xi f(\mathbf{x})+\sum_{\alpha\in R}\frac{k(\alpha)}2\langle\alpha,\xi\rangle\frac{f(\mathbf{x})\!-\!f(\sigma_\alpha(\mathbf{x}))}{\langle\alpha,\mathbf{x}\rangle}\\
&=\partial_\xi f(\mathbf{x})+\hspace{-1mm}\sum_{\alpha\in R^+}\hspace{-1mm}k(\alpha)\langle\alpha,\xi\rangle\frac{f(\mathbf{x})\!-\!f(\sigma_\alpha(\mathbf{x}))}{\langle\alpha,\mathbf{x}\rangle}.
\end{align*}
Here $R^+$ is any fixed positive subsystem of $R$. The Dunkl operators $T_\xi$, which were introduced in~\cite{Dunkl}, commute pairwise and are skew-symmetric with respect to the $G$-invariant measure $dw(\mathbf{x})=w(\mathbf{x})\,d\mathbf{x}$, where
\begin{equation}\label{eq:measure}
w(\mathbf{x})
=\prod_{\alpha\in R}|\langle \alpha, \mathbf{x}\rangle|^{k(\alpha)}
=\hspace{-1mm}\prod_{\alpha\in R^+}\hspace{-1mm}|\langle \alpha, \mathbf{x}\rangle|^{2k(\alpha)}.
\end{equation}
Set $T_j=T_{e_j}$, where $\{e_1,\dots,e_N\}$ is the canonical basis of $\mathbb{R}^N$.

Let us denote a Euclidean ball centered at $\mathbf{x} \in \mathbb{R}^N$ of radius $R$ by $B(\mathbf{x},R)$. It is known (see e.g.~\cite{conjugate}) that
\begin{equation}\label{eq:measure_of_ball_formula}
w(B(\mathbf{x},r))\sim r^{N}\prod_{\alpha\in R}(\,|\langle\alpha,\mathbf{x}\rangle|+r\,)^{k(\alpha)},
\end{equation}
hence
\begin{equation}\label{eq:measure_of_balls}
C^{-1}\Big(\frac{R}{r}\Big)^{N}\leq\frac{w(B(\mathbf{\mathbf{x}},R))}{w(B(\mathbf{\mathbf{x}},r))}\leq C \Big(\frac{R}{r}\Big)^{\mathbf N}
\end{equation}
for any $\mathbf{x} \in \mathbb{R}^N$ and $R \geq r>0$. Here and subsequently,
\begin{equation}
\mathbf{N}=N+\sum_{\alpha \in R^{+}}2k(\alpha).
\end{equation}
The number $\mathbf{N}$ is called the homogeneous dimension. The measure $w$ is doubling i.e.,
\begin{equation}
w(B(\mathbf{x},R)) \asymp w(B(\mathbf{x},2R)).
\end{equation}

We define the \textit{distance of the orbits} associated with the group $G$ by
\begin{equation}
\dz(\mathbf{x},\mathbf{y})=\min_{g \in G}\|\mathbf{x}-g(\mathbf{y})\|
\end{equation}
for $\mathbf{x},\mathbf{y} \in \mathbb{R}^N$, where $\|\mathbf{x}-g(\mathbf{y})\|$ is the Euclidean distance of $\mathbf{x}$ and $g(\mathbf{y})$.

For any ball $B(\mathbf{x},R)$, we write
\begin{equation}
\mathcal{O}(B(\mathbf{x},R))=\bigcup_{g \in G}B(g(\mathbf{x}),R).
\end{equation}
Clearly,
\begin{equation}
w(B(\mathbf{x},R)) \leq w(\mathcal{O}(B(\mathbf{x},R))) \leq |G|w(B(\mathbf{x},R)).
\end{equation}

In the paper we consider the \textit{Dunkl harmonic oscillator}
\begin{equation}
L f(\mathbf{x})=\Delta f(\mathbf{x})-\|\mathbf{x}\|^2f(\mathbf{x}),
\end{equation}
where
\begin{equation}
\Delta f(\mathbf{x})=\sum_{j=1}^{N}T_{j}^2f(\mathbf{x})
\end{equation}
is the \textit{Dunkl Laplacian}.
It is known that the operators $\Delta$ and $L $, initially defined on the Schwartz class $\mathcal S(\mathbb R^N)$, have extensions to  the infitesimal generators of the Dunkl heat semigroup ${H_{t}}=e^{t\Delta}$ and the Dunkl--Hermite semigroup $K_t=e^{tL }$, respectively (see, e.g., \cite{Roesler2}, \cite{Rosler} and \cite{Roesler-Voit}). These semigroups act by linear self--adjoint operators on $L^2(\mathbb{R}^N,dw)$ and by linear contractions on $L^p(\mathbb{R}^N,dw)$ for  $1 \leq p \leq \infty$.

Set
\begin{equation}
\label{eq:rho}
\rho(\mathbf{x})=\frac{1}{1+\|\mathbf{x}\|}.
\end{equation}

We are in a position to state our main results.

\subsection{Characterization by atomic decomposition}

\begin{definition}\label{def:atomic_Dunkl_Hermite}\normalfont
Let $A>1$ be a constant large enough, which will be chosen later. A measurable function $a(\mathbf{x})$ is called an \textit{atom for the Hardy space} $\mathcal{H}^{1,{{\rm{at}}}}$  associated with the Dunkl harmonic oscillator $L $ if
\begin{enumerate}[(A)]
\item{$\supp a \subset B(\mathbf{x}_0,r)$ for some $\mathbf{x}_0 \in \mathbb{R}^N$,}\label{numitem:supports}
\item{$\sup_{\mathbf{y} \in \mathbb{R}^N}|a(\mathbf{y})|\leq w(B(\mathbf{x}_0,r))^{-1}$,}\label{numitem:size}
\item{$r \leq A\rho(\mathbf{x}_0)$,}\label{numitem:bounded_radius}
\item{if $r < \rho(\mathbf{x}_0)$, then $\int_{\mathbb{R}^N}a(\mathbf{x})\,dw(\mathbf{x})=0$.}\label{numitem:cancellation}
\end{enumerate}
 The \textit{atomic Hardy space} $\mathcal{H}^{1,{{\rm{at}}}}$ assosiated with the Dunkl harmonic oscillator is the space of functions $f \in L^1(\mathbb{R}^N,dw)$ which admit a representation of the form
\begin{equation}
f(\mathbf{x})=\sum_{j=1}^{\infty}c_j a_j(\mathbf{x}),
\end{equation}
where $c_j \in \mathbb{C}$ and $a_j$ are atoms for the Hardy space $\mathcal{H}^{1,{{\rm{at}}}}$ such that $\sum_{j=1}^{\infty}|c_j|<\infty$. The space $\mathcal{H}^{1,{{\rm{at}}}}$ is a Banach space with the norm
\begin{equation}
\|f\|_{\mathcal{H}^{1,{{\rm{at}}}}}=\inf\left\{\sum_{j=1}^{\infty}|c_j|:  f(\mathbf{x})=\sum_{j=1}^{\infty}c_j a_j(\mathbf{x}) \text{ and } a_j \text{ are }\mathcal{H}^{1,{{\rm{at}}}}\text{ atoms}\right\}.
\end{equation}

\end{definition}
We are now in a position to state our first main results.
\begin{theorem}\label{teo:atomic_Dunkl_Hermite}
The spaces $\mathcal{H}^1$ and $\mathcal{H}^{1,{{\rm{at}}}}$ coincide and their norms are equivalent.
\end{theorem}

Let us remark that we have two type of atoms for ${\mathcal{H}^{1,{{\rm{at}}}}}$. All of them are supported by $B(\mathbf x, r)$ with $r<A\rho (\mathbf x)$ and satisfy the size condition~\eqref{numitem:size}. However, if $r<\rho (x)$, then the additional condition~\eqref{numitem:cancellation} is required. Let us also note that the atoms for ${\mathcal{H}^{1,{{\rm{at}}}}}$ have a similar type of localization, given by the function $\rho$, to those for the Hardy space associated with the classical harmonic oscillator (cf. \cite{DzZien},~\cite{Dziubanski97}).

\subsection{Characterization by the Riesz transform}
\begin{definition}\label{def:Riesz_Dunkl_Hermite}\normalfont
For $j=1,2,\ldots,N$, we define the  \textit{Riesz transform} $\widetilde{R}_j$ associated with the Dunkl harmonic oscillator by
\begin{equation}
\widetilde{R}_jf(\mathbf{x})=T_jL ^{-1\slash 2}=d_k\int_{0}^{\infty}T_jK_tf(\mathbf{x})\,\frac{dt}{\sqrt{t}},
\end{equation}
where $d_k=2^{\mathbf{N}/2}\Gamma((\mathbf{N}+1)/2)\pi^{-1/2}$.
\end{definition}

It was shown by Amri~\cite{Amri15} that $\widetilde{R}_j$ are bounded operators on $L^p(\mathbb{R}^N,dw)$ for $1<p<\infty$, and of weak type $(1,1)$.

The second goal of this paper is about characterization of $\mathcal{H}^1$ by $\widetilde{R}_j$.
\begin{theorem}\label{teo:Riesz_Hermite}
Let $f \in L^1(\mathbb{R}^N,dw)$. Then $f \in \mathcal{H}^1$ if and only if $\widetilde{R}_j f \in L^1(\mathbb{R}^N,dw)$ for $j=1,2,\ldots,N$. Moreover,
\begin{equation}
\label{eq:Riesz}
\|f\|_{\mathcal{H}^1} \lesssim \|f\|_{L^1(\mathbb{R}^N,dw)}+\sum_{j=1}^{N}\|\widetilde{R}_jf\|_{L^1(\mathbb{R}^N,dw)} \lesssim \|f\|_{\mathcal{H}^1}\footnote{As usual, the symbol $\lesssim$ means that there exists a constant $C>0$ such that $\|f\|_{\mathcal{H}^1} \leq C \left(\|f\|_{L^1(\mathbb{R}^N,dw)}+\sum_{j=1}^{N}\|\widetilde{R}_jf\|_{L^1(\mathbb{R}^N,dw)}\right)$.}.
\end{equation}
\end{theorem}

\section{Heat  kernels for $\Delta$ and $L $}
In this section we collect basic facts concerning integral kernels associated with the Dunkl Laplacian  and Dunkl harmonic oscillator. Then we prove some estimates for them which will be used latter on.

For fixed $\mathbf{x} \in\mathbb{R}^N$, the \textit{Dunkl kernel} $\mathbf{y} \to E(\mathbf{x},\mathbf{y})$ is the unique solution to the system
\begin{align*}
\begin{cases}
\,T_\xi f=\langle\xi,\mathbf{x} \rangle\,f\quad \text{ for all }\;\xi\in\mathbb{R}^N,\\
\;f(0)=1.
\end{cases}
\end{align*}
It is  known (see, e.g., \cite{Roesler2}, \cite{Rosler} and \cite{Roesler-Voit}) that the semigroups ${H_t}=e^{t\Delta}$ and $K_t=e^{tL }$ have the form
\begin{equation}
{H_t}f(\mathbf{x})=\int_{\mathbb{R}^N}{h_t}(\mathbf{x},\mathbf{y})f(\mathbf{y})\,dw(\mathbf{y}),
\end{equation}
\begin{equation}
K_tf(\mathbf{x})=\int_{\mathbb{R}^N}k_t(\mathbf{x},\mathbf{y})f(\mathbf{y})\,dw(\mathbf{y}),
\end{equation}
where
\begin{equation}\label{eq:Dunkl}
{h_t}(\mathbf{x},\mathbf{y})=c_{k}^{-1}2^{-\frac{\mathbf{N}}{2}}t^{-\frac{\mathbf{N}}{2}}\exp\left(-\frac{\|\mathbf{x}\|^2+\|\mathbf{y}\|^2}{4t}\right)E\left(\frac{\mathbf{x}}{\sqrt{2t}},\frac{\mathbf{y}}{\sqrt{2t}}\right),
\end{equation}
and
\begin{equation}\label{eq:Dunkl_Hermite}
k_t(\mathbf{x},\mathbf{y})=c_{k}^{-1}2^{-\frac{\mathbf{N}}{2}}t_1^{-\frac{\mathbf{N}}{2}}\exp\left(-\frac{\coth(2t)(\|\mathbf{x}\|^2+\|\mathbf{y}\|^2)}{2}\right)E\left(\frac{\mathbf{x}}{\sqrt{2t_1}},\frac{\mathbf{y}}{\sqrt{2t_1}}\right).
\end{equation}
Here and subsequently,
\begin{equation}\label{eq:t1}
 t_1=\frac{1}{2}\sinh(2t),
\end{equation}
\begin{equation}
c_k=\int_{\mathbb{R}^N}\exp\left(-\frac{\|\mathbf{x}\|^2}{2}\right)\,dw(\mathbf{x}).
\end{equation}
It is easily seen that
\begin{equation}\label{eq:connection_Dunkl_Dunkl_Hermite}
k_t(\mathbf{x},\mathbf{y})=\exp\left(\left(\frac{1}{4t_1}-\frac{1}{2}\coth(2t)\right)(\|\mathbf{x}\|^2+\|\mathbf{y}\|^2)\right){h_{t_1}}(\mathbf{x},\mathbf{y}).
\end{equation}
The heat kernel ${h_t}(\mathbf{x},\mathbf{y})$ has the following scaling and symmetry properties which are consequences of~\eqref{eq:Dunkl} and~\eqref{eq:Dunkl_Hermite}:
\begin{equation}\label{eq:symmetry}
{h_t}(\mathbf{x},\mathbf{y})={h_t}(\mathbf{y},\mathbf{x})\text{  for all }\mathbf{x},\mathbf{y} \in \mathbb{R}^N,
\end{equation}
\begin{equation}\label{eq:rescaling}
{h_{\lambda^2 t}}(\lambda \mathbf{x},\lambda \mathbf{y})=|\lambda|^{-\mathbf{N}}{h_t}(\mathbf{x},\mathbf{y})\text{  for all }\mathbf{x},\mathbf{y} \in \mathbb{R}^N, \lambda>0.
\end{equation}
Set
\begin{equation}
\mathcal{G}_t(\mathbf{x},\mathbf{y})=\frac{1}{w(B(\mathbf{x},\sqrt{t}))}\sum_{g \in G}\exp\left(\frac{-\|\mathbf{x}-g(\mathbf{y})\|^2}{t} \right).
\end{equation}
Obviously,
\begin{equation} \label{eq:equivalent_G_t}
\mathcal{G}_t(\mathbf{x},\mathbf{y}) \asymp \frac{1}{w(B(\mathbf{x},\sqrt{t}))}\exp\left(\frac{-\dz(\mathbf{x},\mathbf{y})^2}{t} \right).
\end{equation}
The following theorem was proved in~\cite[Theorem 4.3]{conjugate} and~\cite[Theorem 3.1]{Dziubanski18}.

\begin{theorem}\label{teo:theoremGauss}
\begin{enumerate}[(A)]
\item{
Gaussian type bounds: there are constants $C,c>0$ such that
\begin{align*}
0<{h_t}(\mathbf{x},\mathbf{y}) \leq C \left(\frac{\|\mathbf{x}-\mathbf{y}\|}{\sqrt{t}}+1\right)^{-2}\mathcal{G}_{t/c}(\mathbf{x},\mathbf{y}),
\end{align*}
for every $t>0$ and for every $\mathbf{x},\mathbf{y} \in \mathbb{R}^N$.
}\label{numitem:Gauss}
\item{
Time derivatives:  for any integer $m \geq 0$ there are constants $C,C',c>0$ such that
\begin{align*}
\left|\partial_t^{m} {h_t}(\mathbf{x},\mathbf{y})\right|\leq Ct^{-m} {h_{2t}}(\mathbf{x},\mathbf{y})
\leq C't^{-m} \mathcal{G}_{t/c}(\mathbf{x},\mathbf{y}),
\end{align*}
for every $t>0$ and for every $\mathbf{x},\mathbf{y} \in \mathbb{R}^N$.
}\label{numitem:time}
\item{
Space and time derivatives: for any integer $m \geq 0$, there are constants $C,c>0$ such that
\begin{align*}
\left|T_j \partial_t^{m} {h_t}(\mathbf{x},\mathbf{y})\right|\leq C t^{-m-1\slash 2} \mathcal{G}_{t/c}(\mathbf{x},\mathbf{y}),
\end{align*}
for every $t>0$, $j=1,2,\ldots,N$, and for every $\mathbf{x},\mathbf{y} \in \mathbb{R}^N$.
}\label{numitem:space}
\end{enumerate}
\end{theorem}

\section{Hardy spaces for the Dunkl Laplacian.}
The Hardy spaces $H^1$ in the Dunkl setting were studied in ~\cite{conjugate}. In the present section we state three equivalent definitions of $H^1$: by means of a nontangential maximal function, by an atomic decomposition, and by Riesz transforms. We shall also give a short proof of characterization by radial maximal function. These definitions  will be used to prove characterizations of local Hardy spaces in the Dunkl setting.
\begin{definition}\normalfont
Fix a normalized  root system $R$ and a multiplicity function $k\geq 0$. We define the \textit{Hardy space} $H^{1}$ associated with $\Delta$   to be
 $$H^1=\{f \in L^{1}(\mathbb{R}^N,dw): \mathcal M^*f\in L^{1}(\mathbb{R}^N,dw)\},$$
 where
\begin{equation}
\mathcal{M}^{*}f(\mathbf{x})=\sup_{t>0,\,\|\mathbf{x}-\mathbf{y}\|<\sqrt{t}}|{H_t}f(\mathbf{x})|
\end{equation}
is the nontangential maximal function for ${H_t}$.
The Hardy space $H^1$ is a Banach space with the norm
\begin{equation}
\|f\|_{H^{1}}=\|\mathcal{M}^{*}f\|_{L^{1}(\mathbb{R}^N,dw)}.
\end{equation}
\end{definition}

\subsection{Characterization of $H^1$ by radial maximal function.}
Our goal in this subsection is to show that the Hardy space $H^1$ is characterized by the following radial maximal function
$$\mathcal{M}_Rf(\mathbf{x})=\sup_{t>0}|{H_{t}}f(\mathbf{x})|$$
with respect to ${H_t}$.

Let
\begin{align*}
\mathcal{M}f(\mathbf{x})=\sup_{\mathbf{x} \in B}\frac{1}{w(B)}\int_{B}|f(\mathbf{y})|\,dw(\mathbf{y}),
\end{align*}
where the supremum is taken over all Euclidean balls $B$ which contain $\mathbf{x}$, be the Hardy-Littlewood maximal function.
We shall need the following tangential maximal function
$$\mathcal{M}^{**}f(\mathbf{x})=\sup_{\mathbf{y} \in \mathbb{R}^N,\,t>0}\left(1+\frac{\dz(\mathbf{x},\mathbf{y})}{\sqrt{t}}\right)^{-2\mathbf{N}}|{H_{t}}f(\mathbf{y})|.$$
Obviously, $\mathcal{M}_{R}f(\mathbf{x}) \leq \mathcal{M}^{*}f(\mathbf{x}) \lesssim \mathcal{M}^{**}f(\mathbf{x})$.

\begin{proposition}
\label{propo:radial_nontangential}
There is a constant $C>0$ such that for any $f \in L^1(\mathbb{R}^N,dw)$ we have
\begin{align*}
\|\mathcal{M}_{R}f\|_{L^1(\mathbb{R}^N,dw)} \leq \|\mathcal{M}^{*}f\|_{L^1(\mathbb{R}^N,dw)} \leq C\|\mathcal{M}_{R}f\|_{L^1(\mathbb{R}^N,dw)} .
\end{align*}
\end{proposition}

To prove the proposition above, we need some lemmas.

\begin{lemma}
\label{lem:HardyLittlewood_maximal}
Let $\alpha>0$. For any $t>0$ and $\mathbf{x} \in \mathbb{R}^N$ we have
\begin{align*}
\int_{\mathbb{R}^N} \frac{1}{w(B(\mathbf{y},\sqrt{t}))}\left(1+\frac{\dz(\mathbf{x},\mathbf{y})}{\sqrt{t}}\right)^{-\mathbf{N}-\alpha}|f(\mathbf{y})|\,dw(\mathbf{y}) \lesssim \sum_{g \in G}\mathcal{M}f(g(\mathbf{x})).
\end{align*}
\end{lemma}

\begin{proof}
The proof is standard but we provide this for the sake of completeness. We have
\begin{align*}
&\int_{\mathbb{R}^N} \frac{1}{w(B(\mathbf{y},\sqrt{t}))}\left(1+\frac{\dz(\mathbf{x},\mathbf{y})}{\sqrt{t}}\right)^{-\mathbf{N}-\alpha}|f(\mathbf{y})|\,dw(\mathbf{y})\\=&\int_{\mathcal{O}(B(\mathbf{x},\sqrt{t}))} \frac{1}{w(B(\mathbf{y},\sqrt{t}))}\left(1+\frac{\dz(\mathbf{x},\mathbf{y})}{\sqrt{t}}\right)^{-\mathbf{N}-\alpha}|f(\mathbf{y})|\,dw(\mathbf{y})\\&+\sum_{n=0}^{\infty}\int_{2^n\sqrt{t} \leq \dz(\mathbf{x},\mathbf{y}) < 2^{n+1}\sqrt{t}} \frac{1}{w(B(\mathbf{y},\sqrt{t}))}\left(1+\frac{\dz(\mathbf{x},\mathbf{y})}{\sqrt{t}}\right)^{-\mathbf{N}-\alpha}|f(\mathbf{y})|\,dw(\mathbf{y})\\ \lesssim
&\sum_{g \in G}\mathcal{M}f(g(\mathbf{x}))+\sum_{n=0}^{\infty}2^{-n(\mathbf{N}+\alpha)}\int_{\mathcal{O}(B(\mathbf{x},2^{n+1}\sqrt{t}))} \frac{1}{w(B(\mathbf{y},\sqrt{t}))}|f(\mathbf{y})|\,dw(\mathbf{y})\\\lesssim & \sum_{g \in G}\mathcal{M}f(g(\mathbf{x}))+\sum_{g \in G}\sum_{n=0}^{\infty}2^{-n\alpha}\mathcal{M}f(g(\mathbf{x})),
\end{align*}
where in the inequalities above we have used the doubling property of $w$ (see~\eqref{eq:measure_of_balls}).
\end{proof}

\begin{lemma}
\label{lem:M**_finite_ae}
There exists a constant $C>0$ such that for any $f \in L^1(\mathbb{R}^N,dw)$ we have
\begin{align*}
\mathcal{M}^{**}f(\mathbf{x}) \leq C \sum_{g \in G}\mathcal{M}f(g(\mathbf{x})).
\end{align*}
\end{lemma}

\begin{proof}
Fix $t>0$ and $\mathbf{x} \in \mathbb{R}^N$. By Theorem~\ref{teo:theoremGauss}~\eqref{numitem:Gauss}, we have
\begin{equation}
\label{eq:ae_calculation}
\begin{split}
&\left(1+\frac{\dz(\mathbf{x},\mathbf{y})}{\sqrt{t}}\right)^{-2\mathbf{N}}|{H_t}f(\mathbf{y})|\\
&\lesssim\left(1+\frac{\dz(\mathbf{x},\mathbf{y})}{\sqrt{t}}\right)^{-2\mathbf{N}}\int\frac{1}{w(B(\mathbf{z},\sqrt{t}))}\left(1+\frac{\dz(\mathbf{z},\mathbf{y})}{\sqrt{t}}\right)^{-2\mathbf{N}}|f(\mathbf{z})|\,dw(\mathbf{z}).
\end{split}
\end{equation}
By an elementary calculation, we see that
\begin{align*}
\left(1+\frac{\dz(\mathbf{x},\mathbf{y})}{\sqrt{t}}\right)^{-2\mathbf{N}}\left(1+\frac{\dz(\mathbf{z},\mathbf{y})}{\sqrt{t}}\right)^{-2\mathbf{N}} \lesssim \left(1+\frac{\dz(\mathbf{x},\mathbf{z})}{\sqrt{t}}\right)^{-2\mathbf{N}}.
\end{align*}
Hence, using Lemma~\ref{lem:HardyLittlewood_maximal} with $\alpha=\mathbf{N}$, we get
\begin{align*}
\left(1+\frac{\dz(\mathbf{x},\mathbf{y})}{\sqrt{t}}\right)^{-2\mathbf{N}}|{H_t}f(\mathbf{y})|\lesssim \sum_{g \in G}\mathcal{M}f(g(\mathbf{x})).
\end{align*}
\end{proof}

\begin{corollary}
\label{coro:finite_ae}
Since $\mathcal{M}$ is of weak type $(1,1)$, $\mathcal{M}^{**}f$ is finite almost everywhere for any $f \in L^1(\mathbb{R}^N,dw)$.
\end{corollary}

\begin{lemma}
\label{lem:nontangential_radial}
There exists a constant $C>0$ such that for every $f \in L^1(\mathbb{R}^N,dw)$ we have
\begin{align*}
\|\mathcal{M}^{**}f\|_{L^1(\mathbb{R}^N,dw)} \leq C \|\mathcal{M}_{R}f\|_{L^1(\mathbb{R}^N,dw)}.
\end{align*}
\end{lemma}
\begin{proof}
The proof uses standard arguments (cf.~\cite{Grafakos}). Thanks to the semigroup property is simple  and short.  For the convenience of the reader we provide details.  Fix $t>0$ and $\mathbf{x} \in \mathbb{R}^N$. Similarly as in the proof of Lemma~\ref{lem:M**_finite_ae}, we have
\begin{align*}
&\left(1+\frac{\dz(\mathbf{x},\mathbf{y})}{\sqrt{t}}\right)^{-2\mathbf{N}}|{H_t}f(\mathbf{y})|=\left(1+\frac{\dz(\mathbf{x},\mathbf{y})}{\sqrt{t}}\right)^{-2\mathbf{N}}|{H_{t/2}}{H_{t/2}}f(\mathbf{y})|\\&\lesssim\int\frac{1}{w(B(\mathbf{z},\sqrt{t}))}\left(1+\frac{\dz(\mathbf{x},\mathbf{z})}{\sqrt{t}}\right)^{-2\mathbf{N}}|{H_{t/2}}f(\mathbf{z})|\,dw(\mathbf{z})\\&\lesssim \left(\mathcal{M}^{**}f(\mathbf{x})\right)^{1/4}\int\frac{1}{w(B(\mathbf{z},\sqrt{t}))}\left(1+\frac{\dz(\mathbf{x},\mathbf{z})}{\sqrt{t}}\right)^{-3\mathbf{N}/2}|{H_{t/2}}f(\mathbf{z})|^{3/4}\,dw(\mathbf{z})\\&\lesssim \left(\mathcal{M}^{**}f(\mathbf{x})\right)^{1/4} \sum_{g \in G}\mathcal{M}((\mathcal{M}_{R}f(\cdot))^{3/4})(g(\mathbf{x})),
\end{align*}
where in the last inequality we have used Lemma~\ref{lem:HardyLittlewood_maximal} with $\alpha=\mathbf{N}/2$. In summary, we have obtained
\begin{align*}
\mathcal{M}^{**}f(\mathbf{x}) \lesssim \left(\mathcal{M}^{**}f(\mathbf{x})\right)^{1/4} \sum_{g \in G}\mathcal{M}((\mathcal{M}_{R}f(\cdot))^{3/4})(g(\mathbf{x})).
\end{align*}
Since $\mathcal{M}^{**}f$ is finite almost everywhere (see Corollary~\ref{coro:finite_ae}), we have
\begin{align*}
\mathcal{M}^{**}f(\mathbf{x}) \lesssim \sum_{g \in G}\left(\mathcal{M}((\mathcal{M}_{R}f(\cdot))^{3/4})(g(\mathbf{x}))\right)^{4/3}
\end{align*}
for a.e.  $\mathbf{x} \in \mathbb{R}^N$. Since $\mathcal{M}$ is bounded on $L^{4/3}(\mathbb{R}^N,dw)$, the lemma follows.
\end{proof}
Proposition~\ref{propo:radial_nontangential} is an easy consequence of Lemma~\ref{lem:nontangential_radial}.

\subsection{Atomic decomposition of $H^1$.}

\begin{definition}\normalfont
A function {$a(\mathbf{x})$} is said to be  an  {atom} for $H^{1,{{\rm{at}}}}$  if
{there exist} a ball {$B$} such that
\begin{itemize}
\item
$\supp a \subset B$\,,
\vspace{.5mm}
\item
$\|a\|_{L^{\infty}} \leq w(B)^{-1}$
\item
$\int_{\mathbb{R}^N}a(\mathbf{x})\,dw(\mathbf{x})=0$.
\end{itemize}
A function $f$ belongs to ${H^{1,{{\rm{at}}}}} $ if there are $c_j\in\mathbb{ C}$ and atoms $a_j$ for ${H^{1,{{\rm{at}}}}}$ such that $\sum_{j=1}^{\infty}|c_j|<\infty$,
\begin{equation}\label{eq:atomic_representation}
f=\sum_{j=1}^{\infty}c_j\,a_j\,.
\end{equation}
{In this case,} set
$
\|f\|_{H^{1,{{\rm{at}}}}}=\inf\,\Bigl\{\,\sum_{j=1}^{\infty}|c_j|\,\Bigr\}\,,
$
where the infimum is taken over all representations {\eqref{eq:atomic_representation}}.
\end{definition}

The following theorem was proved in~\cite[Theorem 1.6]{Dziubanski18}.

\begin{theorem}\label{teo:maximal_atom}
The spaces $H^{1,{{\rm{at}}}}$ and $H^1$ coincide and their norms are equivalent, i.e. there exists a constant $C>0$ such that
\begin{equation}
C^{-1}\|f\|_{H^{1,{{\rm{at}}}}} \leq \|f\|_{H^{1}} \leq C\|f\|_{H^{1,{{\rm{at}}}}}.
\end{equation}
\end{theorem}

\subsection{Riesz transform characterization of $H^1$}

\begin{definition}\normalfont
Let $T>0$. For $j=1,2,\ldots,N$ we define the \textit{Riesz transform} $R_jf$ of $f$
associated with the Dunkl Laplacian $\Delta$ by
\begin{equation}
R_{j}f(\mathbf{x})=d_k\int_{0}^{\infty}T_j {H_{t}}f(\mathbf{x})\,\frac{dt}{\sqrt{t}},
\end{equation}
where $d_k=2^{\mathbf{N}/2}\Gamma((\mathbf{N}+1)/2)\pi^{-1/2}$.
\end{definition}
It is well known that $R_j$ are bounded operators on $L^p(dw)$ for $1<p<\infty$ (see~\cite{AS}). The action of $R_j$ on $L^1(\mathbb{R}^N,dw)$ functions, thanks to Theorem~\ref{teo:theoremGauss}~\eqref{numitem:space}, is defined in a weak sense (see~\cite[Section 8]{conjugate}).
The following theorem was proved in~\cite[Theorem 2.11]{conjugate}.
\begin{theorem}\label{teo:Riesz}
Let $f \in L^1(\mathbb{R}^N,dw)$. Then $f \in H^1$ if and only if $R_jf \in L^1(\mathbb{R}^N,dw)$ for $j=1,2,\ldots,N$. Moreover,
\begin{equation}
\|f\|_{H^1} \lesssim \|f\|_{L^1(\mathbb{R}^N,dw)}+\sum_{j=1}^{N}\|R_jf\|_{L^1(\mathbb{R}^N,dw)} \lesssim \|f\|_{H^1}.
\end{equation}
\end{theorem}

\section{Local Hardy spaces for the Dunkl Laplacian.}\label{sec:localHardy}
In this section we introduce family of local Hardy spaces  ${H^{1}_{{\rm{loc}},T}}$ associated with the Dunkl Laplacian. Our starting definition is that by means of a local maximal function. Then we characterize  ${H^{1}_{{\rm{loc}},T}}$ by relevant local Riesz transforms and atomic decompositions.

\subsection{Definition of local Hardy space  by the maximal function. Relation with $H^1$.}

\begin{definition}\normalfont
Let $T>0$ and $f \in L^{1}(\mathbb{R}^N,dw)$. We say that $f$ belongs to the \textit{local Hardy space} $H_{\rm{loc},T}^{1}$ associated with the Dunkl Laplacian if and only  if
\begin{equation}
f_{{\rm{loc}},T}^{*}(\mathbf{x})=\sup_{0<t \leq T^2}|{H_t}f(\mathbf{x})|
\end{equation}
belongs to $L^1(\mathbb{R}^N,dw)$. The norm in the space is given by
\begin{equation}
\|f\|_{H^{1}_{{\rm{loc}},T}}=\|f^{*}_{{\rm{loc}},T}\|_{L^{1}(\mathbb{R}^N,dw)}.
\end{equation}
\end{definition}

The remaining part of this subsection is devoted to relations of  ${H^{1}_{{\rm{loc}},T}}$ with $H^1$. This is stated in Proposition \ref{propo:H_12_I_characterization}. We begin by proving auxiliary lemmas.
\begin{lemma}\label{lem:substitute_t_by_dz}
Let $c>0$ and $\alpha \in \mathbb{R}$ be such that $N+2\alpha \geq 0$. There is $C>0$ such that for any $t>0$, $\mathbf{x},\mathbf{y} \in \mathbb{R}^N$  satisfying $\dz(\mathbf{x},\mathbf{y}) \neq 0$ we have
\begin{equation}
\label{eq:substitute_1}
\frac{1}{t^{\alpha}}\mathcal{G}_{t/c}(\mathbf{x},\mathbf{y}) \leq C\frac{1}{(\dz(\mathbf{x},\mathbf{y}))^{2\alpha}}\frac{1}{w(B(\mathbf{x},\dz(\mathbf{x},\mathbf{y})))}.
\end{equation}
Moreover, if $0<t<T$, then
\begin{equation}
\label{eq:substitute_2}
\frac{1}{t^{\alpha}}\mathcal{G}_{t/c}(\mathbf{x},\mathbf{y}) \leq C\frac{1}{(\dz(\mathbf{x},\mathbf{y}))^{2\alpha}}\frac{1}{w(B(\mathbf{x},\dz(\mathbf{x},\mathbf{y})))}\exp\left(-c\frac{\dz(\mathbf{x},\mathbf{y})^2}{2T}\right).
\end{equation}
\end{lemma}

\begin{proof}
\textbf{Case 1.} $\sqrt{t} \leq \dz(\mathbf{x},\mathbf{y})$. Using~\eqref{eq:measure_of_balls}, we have
\begin{equation}
\begin{split}
\frac{1}{t^{\alpha}}\mathcal{G}_{t/c}(\mathbf{x},\mathbf{y}) &\lesssim \frac{1}{t^{\alpha}}\frac{1}{w(B(\mathbf{x},\dz(\mathbf{x},\mathbf{y})))}\frac{w(B(\mathbf{x},\dz(\mathbf{x},\mathbf{y})))}{w(B(\mathbf{x},\sqrt{t}))}\exp\left(-c\frac{\dz(\mathbf{x},\mathbf{y})^2}{t}\right)\\ &\lesssim \frac{1}{t^{\alpha}}\frac{1}{w(B(\mathbf{x},\dz(\mathbf{x},\mathbf{y})))}\left(\frac{\dz(\mathbf{x},\mathbf{y})}{\sqrt{t}}\right)^{\mathbf{N}}\exp\left(-c\frac{\dz(\mathbf{x},\mathbf{y})^2}{t}\right).
\end{split}
\end{equation}
Finally, the last term is estimated by
\begin{equation}
\label{eq:doubling_eq}
\begin{split}
&\frac{1}{t^{\alpha}}\frac{1}{w(B(\mathbf{x},\dz(\mathbf{x},\mathbf{y})))}\left(\frac{\dz(\mathbf{x},\mathbf{y})}{\sqrt{t}}\right)^{\mathbf{N}} \left(\frac{\sqrt{t}}{\dz(\mathbf{x},\mathbf{y})}\right)^{\mathbf{N}+2\alpha}\exp\left(-c\frac{\dz(\mathbf{x},\mathbf{y})^2}{2t}\right)\\&=\frac{1}{(\dz(\mathbf{x},\mathbf{y}))^{2\alpha}}\frac{1}{w(B(\mathbf{x},\dz(\mathbf{x},\mathbf{y})))}\exp\left(-c\frac{\dz(\mathbf{x},\mathbf{y})^2}{2t}\right).
\end{split}
\end{equation}
\\
\textbf{Case 2.} $\sqrt{t} \geq \dz(\mathbf{x},\mathbf{y})$. We proceed the same way as in Case 1 with $\mathbf{N}$ replaced by $N$ (it is possible thanks to~\eqref{eq:measure_of_balls}).

In order to prove~\eqref{eq:substitute_2}, we put $t=T$ in~\eqref{eq:doubling_eq}, so the lemma follows.
\end{proof}

\begin{lemma}\label{lem:summation}
Let $T>0$ and $\alpha>0$. There exists a constant $C>0$ such that for any $\mathbf{y} \in \mathbb{R}^N$ we have
\begin{enumerate}[(A)]
\item{$\int_{(\mathcal{O}(B(\mathbf{y},T)))^c}\dz(\mathbf{x},\mathbf{y})^{-2\alpha}\frac{1}{w(B(\mathbf{y},\dz(\mathbf{x},\mathbf{y})))}\,dw(\mathbf{x}) \leq CT^{-2\alpha}$,}\label{numitem:summation_outside}
\item{$\int_{\mathcal{O}(B(\mathbf{y},T))}\dz(\mathbf{x},\mathbf{y})^{2\alpha}\frac{1}{w(B(\mathbf{y},\dz(\mathbf{x},\mathbf{y})))}\,dw(\mathbf{x}) \leq CT^{2\alpha}$.}\label{numitem:summation_inside}
\end{enumerate}
\end{lemma}

\begin{proof}
We will prove only~\eqref{numitem:summation_outside}, the proof of~\eqref{numitem:summation_inside} is similar. By the doubling property of the measure $w$ we have
\begin{align*}
&\int_{(\mathcal{O}(B(\mathbf{y},T)))^c}\dz(\mathbf{x},\mathbf{y})^{-2\alpha}\frac{1}{w(B(\mathbf{y},\dz(\mathbf{x},\mathbf{y})))}\,dw(\mathbf{x}) \\ &\lesssim \sum_{n=0}^{\infty}\int_{2^{n}T \leq \dz(\mathbf{x},\mathbf{y}) < 2^{n+1}T} \dz(\mathbf{x},\mathbf{y})^{-2\alpha}\frac{1}{w(B(\mathbf{y},\dz(\mathbf{x},\mathbf{y})))}\, dw(\mathbf{x})\\ &\lesssim
\sum_{n=0}^{\infty}\int_{2^{n}T \leq \dz(\mathbf{x},\mathbf{y}) < 2^{n+1}T} (2^n T)^{-2\alpha}\frac{1}{w(B(\mathbf{y},2^{n}T))}\, dw(\mathbf{x}) \\ &\lesssim |G|T^{-2\alpha}\sum_{n=0}^{\infty}(2^n)^{-2\alpha}\frac{w(B(\mathbf{y},2^{n+1}T))}{w(B(\mathbf{y},2^{n}T))} \leq CT^{-2\alpha}.
\end{align*}
\end{proof}

Similarly as in the case of classical local Hardy spaces~\cite{Goldberg79}, we  characterize the local Hardy space $H_{{\rm{loc}},T}^1$ in terms of the global Hardy space $H^1$.

\begin{proposition}\label{propo:H_12_I_characterization}
Suppose that $f \in L^{1}(\mathbb{R}^N,dw)$ and $T>0$. Then $f \in H_{{\rm{loc}},T}^1$ if and only if $({H_{T^2/2}}-I)f \in H^1$. Moreover,
\begin{equation}
\|f\|_{H^{1}_{{\rm{loc}},T}} \lesssim \|({H_{T^2/2}}-I)f\|_{H^{1}}+\|f\|_{L^1(\mathbb{R}^N,dw)} \lesssim \|f\|_{H^{1}_{{\rm{loc}},T}}
\end{equation}
uniformly in $T$.
\end{proposition}

\begin{proof}
Assume that $f \in H_{{\rm{loc}},T}^1$. We have
\begin{align*}
\mathcal{M}_R\left(({H_{T^2/2}}-I)f\right)(\mathbf{x}) &\leq \sup_{0<t\leq T^2/2}|{H_t}({H_{T^2/2}}-I)f(\mathbf{x})|+\sup_{t>T^2/2}|{H_t}({H_{T^2/2}}-I)f(\mathbf{x})|
\\&=F_{1}(\mathbf{x})+F_{2}(\mathbf{x}).
\end{align*}
By assumption, $\|F_1\|_{L^1(\mathbb{R}^N,dw)} \leq 2 \|f\|_{H^{1}_{{\rm{loc}},T}}$. We shall prove that $\|F_2\|_{L^1(\mathbb{R}^N,dw)} \lesssim \|f\|_{L^1(\mathbb{R}^N,dw)}$. Since
\begin{align*}
&{H_t}({H_{T^2/2}}-I)f(\mathbf{x})=\int_{t}^{t+T^2/2}\partial_{s}{H_s}f(\mathbf{x})\,ds,
\end{align*}
it suffices to show that
\begin{align*}
&\sup_{\mathbf{y} \in \mathbb{R}^N,\, T>0}T^2\int_{\mathbb{R}^N}\sup_{t>T^2/2}\left|\partial_{t}{h_{t}}(\mathbf{x},\mathbf{y})\right|\,dw(\mathbf{x}) < \infty.
\end{align*}
Fix $\mathbf{y} \in \mathbb{R}^N$ and split the integral into two parts: over $\mathcal{O}(B(\mathbf{y},T))$ and over $(\mathcal{O}(B(\mathbf{y},T)))^{c}$. For the first one, by Theorem~\ref{teo:theoremGauss}~\eqref{numitem:time}, we have that the integral is estimated by
\begin{align*}
&T^2\int_{\mathcal{O}(B(\mathbf{y},T))}\sup_{t>T^2/2}\frac{1}{t}\mathcal{G}_{t/c}(\mathbf{x},\mathbf{y})\,dw(\mathbf{x}) \lesssim \int_{\mathcal{O}(B(\mathbf{y},T))}\frac{1}{w(B(\mathbf{y},T))}\,dw(\mathbf{x})\lesssim 1.
\end{align*}
For the integral over $(\mathcal{O}(B(\mathbf{y},T)))^{c}$, we use Theorem~\ref{teo:theoremGauss}~\eqref{numitem:time} again. Finally, the estimation is a direct consequence of Lemma~\ref{lem:substitute_t_by_dz} and Lemma~\ref{lem:summation} with $\alpha=1$.

For the converse implication, suppose that $({H_{T^2/2}}-I)f \in H^1$. We have
\begin{align*}
f^{*}_{{\rm{loc}},T}(\mathbf{x}) \leq \sup_{0<t \leq T^2}|{H_t}({H_{T^2/2}}-I)f(\mathbf{x})|+\sup_{0<t \leq T^2}|{H_{t+T^2/2}}f(\mathbf{x})|=F_3(\mathbf{x})+F_{4}(\mathbf{x}).
\end{align*}
Obviously, $\|F_3\|_{L^1(\mathbb{R}^N,dw)} \lesssim \|({H_{T^2/2}}-I)f\|_{H^1}$. For the purpose of estimating the second summand, we use Theorem~\ref{teo:theoremGauss}~\eqref{numitem:Gauss}. We have
\begin{align*}
&\sup_{0<t \leq T^2}{h_{t+T^2/2}}(\mathbf{x},\mathbf{y}) \lesssim \sup_{0<t \leq T^2}\mathcal{G}_{(t+T^2/2)/c}(\mathbf{x},\mathbf{y}) \lesssim \frac{1}{w\left(B\left(\mathbf{x},T\right)\right)}\exp \left(-\frac{2c\dz(\mathbf{x},\mathbf{y})^2}{3T^2} \right)
\end{align*}
 independently of $\mathbf{y} \in \mathbb{R}^N$, which implies $\|F_4\|_{L^1(\mathbb{R}^N,dw)} \lesssim \|f\|_{L^1(\mathbb{R}^N,dw)}$.
\end{proof}
\subsection{Characterization by the atomic decomposition}

\begin{definition}\label{def:atomic_Dunkl}\normalfont
Let $T>0$. The function $a(\mathbf{x})$ is called an \textit{atom for the local Hardy space} $H^{1,{{\rm{at}}}}_{{\rm{loc}},T}$ if
\begin{enumerate}[(A)]
\item{$\supp a \subset B(\mathbf{x}_0,r)$ for some $\mathbf{x}_0 \in \mathbb{R}^N$ and $r>0$,}\label{numitem:support}
\item{$\sup_{\mathbf{y} \in \mathbb{R}^N}|a(\mathbf{y})| \leq w(B(\mathbf{x}_0,r))^{-1}$,}\label{numitem:L_infty}
\item{If $r < T$, then $\int_{\mathbb{R}^N}a(\mathbf{x})\,dw(\mathbf{x})=0$.}\label{numitem:cancellations}
\end{enumerate}
A function $f$ belongs to the \textit{local Hardy space} $H^{1,{{\rm{at}}}}_{{\rm{loc}},T}$ if there are $c_j\in\mathbb{ C}$ and atoms $a_j$ for ${H^{1,{{\rm{at}}}}}$ such that $\sum_{j=1}^{\infty}|c_j|<\infty$,
\begin{equation}
\label{eq:atomic_representation_local}
f=\sum_{j=1}^{\infty}c_j\,a_j\,.
\end{equation}
{In this case,} set
$
\|f\|_{H^{1,{{\rm{at}}}}_{{\rm{loc}},T}}=\inf\,\Bigl\{\,\sum_{j=1}^{\infty}|c_j|\,\Bigr\}\,,
$
where the infimum is taken over all representations {\eqref{eq:atomic_representation_local}}.
\end{definition}

Let us emphasis that we have two type of atoms for $H^{1,{{\rm{at}}}}_{{\rm{loc}},T}$. If an atom $a$ is supported by $B(\mathbf{x}_0,r)$ with $r\geq T$, then only the size condition~\eqref{numitem:L_infty} is required. However, if $r<T$, then  additionally~\eqref{numitem:cancellations} must hold.

\begin{lemma}\label{lem:atom_in_H1}
There exists a constant $C>0$ such that for any $T>0$ and for any function $a(\mathbf{x})$ such that $\supp a \subset B(\mathbf{p},T)$ for some $\mathbf{p} \in \mathbb{R}^N$ and $\|a\|_{L^{\infty}} \leq w(B(\mathbf{p},T))^{-1}$ the following inequality holds
\begin{equation}
\|a\|_{H^{1}_{{\rm{loc}},T}} \leq C.
\end{equation}
\end{lemma}

\begin{proof}
It is enough to prove that
\begin{align*}
\frac{1}{w(B(\mathbf{p},T))}\int_{\mathbb{R}^N}\sup_{0<t \leq T^2}\int_{B(\mathbf{p},T)}{h_t}(\mathbf{x},\mathbf{y})\, dw(\mathbf{y})\,dw(\mathbf{x}) \leq C,
\end{align*}
where $C>0$ does not depend on $\mathbf{p}$. Split the outer integral into two integrals: over $(\mathcal{O}(B(\mathbf{p},2T)))^{c}$ and over $\mathcal{O}(B(\mathbf{p},2T))$. For the first one, we use Theorem~\ref{teo:theoremGauss}~\eqref{numitem:Gauss}, Lemma~\ref{lem:substitute_t_by_dz} with $\alpha=0$, and the fact that for $\mathbf{y} \in \mathcal{O}(B(\mathbf{p},T))$ and $\mathbf{x} \not\in \mathcal{O}(B(\mathbf{p},2T))$ we have $\dz(\mathbf{x},\mathbf{y})>T$. Therefore, the first integral is bounded by
\begin{align*}
&\frac{C}{w(B(\mathbf{p},T))}\int_{(\mathcal{O}(B(\mathbf{p},T)))^{c}}\int_{B(\mathbf{p},T)}\frac{1}{w(B(\mathbf{x},T))}\exp\left(-c\frac{\dz(\mathbf{x},\mathbf{y})^2}{T^2}\right)\,dw(\mathbf{y})\,dw(\mathbf{x})\\&\leq\frac{C}{w(B(\mathbf{p},T))}\int_{B(\mathbf{p},T)}\int_{\mathbb{R}^N}\frac{1}{w(B(\mathbf{x},T))}\exp\left(-c\frac{\dz(\mathbf{x},\mathbf{y})^2}{T^2}\right)\,dw(\mathbf{x})\,dw(\mathbf{y}).
\end{align*}
Let us recall the fact, that
\begin{equation}\label{eq:integral_of_kernel}
\int_{\mathbb{R}^N}\frac{1}{w(B(\mathbf{x},\sqrt{t}))}\exp\left(-c\frac{\dz(\mathbf{x},\mathbf{y})^2}{t}\right)\,dw(\mathbf{x}) \leq C,
\end{equation}
where $C>0$ is independent of $\mathbf{y}$ and $t$. Using this fact, we see that the estimation for the first part of integral is complete. Thanks to~\eqref{eq:integral_of_kernel}, the second part of the integral is also estimated by constant.
\end{proof}

\begin{proposition}\label{propo:Goldberg}
The spaces $H^{1,{{\rm{at}}}}_{{\rm{loc}},T}$ and $H^1_{{\rm{loc}},T}$ coincide and their norms are equivalent. Moreover, there exists a constant $C>0$ such that for any $T>0$ if $f \in H^{1,{{\rm{at}}}}_{{\rm{loc}},T}$ and $\supp f \subset B(\mathbf{y}_0,T)$, then there are $H^{1,{{\rm{at}}}}_{{\rm{loc}},T}$ atoms $a_j$ such that $\supp a_j \subset B(\mathbf{y}_0,4T)$ and
\begin{equation}
f=\sum_{j=1}^{\infty}c_ja_j, \qquad \sum_{j=1}^{\infty}|c_j| \leq C\|f\|_{H^{1,{{\rm{at}}}}_{{\rm{loc}},T}}.
\end{equation}
\end{proposition}

\begin{proof}
The proof follows the pattern from~\cite{Goldberg79}. For the convenience of the reader, we provide details. We may assume that $T=1$, the general case is a simple consequence of the rescaling property~\eqref{eq:rescaling}. The inclusion $H^{1,{{\rm{at}}}}_{{\rm{loc}},1} \subset H^{1}_{{\rm{loc}},1}$ and the desired norm estimation is due to Lemma~\ref{lem:atom_in_H1} (for atoms supported by   $ B(\mathbf{x},r)$ with $r \geq 1$) and Theorem~\ref{teo:maximal_atom} (if $r<1$).
To prove the inclusion $H^{1}_{{\rm{loc}},1} \subset H^{1,{{\rm{at}}}}_{{\rm{loc}},1}$, we use Proposition~\ref{propo:H_12_I_characterization}. Suppose that $f \in H^{1}_{{\rm{loc}},1}$, then $({H_{1/2}}-I)f \in H^1$. By Theorem~\ref{teo:maximal_atom}, we have
\begin{align*}
&({H_{1/2}}-I)f(\mathbf{x})=\sum_{j=1}^{\infty} c_j a_j(\mathbf{x}),
\end{align*}
where $a_j$ are the atoms for the Hardy space $H^{1,{{\rm{at}}}}$. In particular they are the atoms for the local Hardy space $H^{1,{{\rm{at}}}}_{{\rm{loc}},1}$ and
\begin{align*}
&\sum_{j=1}^{\infty}|c_j| \lesssim \|({H_{1/2}}-I)f\|_{H^1} \lesssim \|f\|_{H^{1}_{{\rm{loc}},1}}.
\end{align*}
Therefore, it is enough to show that ${H_{1/2}}f$ admits a decomposition into $H^{1,{{\rm{at}}}}_{{\rm{loc}},1}$-atoms. Let $\bar{n}=(n_1,n_2,\ldots,n_N)$ for $n_j \in \mathbb{Z}$.
We have

\begin{align*}
{H_{1/2}}f(\mathbf{x})&=\sum_{\bar{n} \in \mathbb{Z}^N}{H_{1/2}}f(\mathbf{x})\chi_{N^{-1/2}\bar{n}+N^{-1/2}[0,1)^N}(\mathbf{x})=\sum_{\bar{n} \in \mathbb{Z}^{N}}c_{\bar{n}} b_{\bar{n}}(\mathbf{x}),
\end{align*}
where
\begin{align*}
b_{\bar{n}}(\mathbf{x})=c_{\bar{n}}^{-1}{H_{1/2}}f(\mathbf{x})\chi_{N^{-1/2}\bar{n}+N^{-1/2}[0,1)^N}(\mathbf{x}),
\end{align*}
\begin{align*}
c_{\bar{n}}=w(B(\bar{n},1))\|{H_{1/2}}f(\cdot)\chi_{N^{-1/2}\bar{n}+N^{-1/2}[0,1)^N}(\cdot)\|_{L^\infty}.
\end{align*}
It is easy to see that $b_{\bar{n}}$ are multiplies of atoms associated with balls of radii $1$. Moreover, by Theorem~\ref{teo:theoremGauss}~\eqref{numitem:Gauss}, we have
\begin{align*}
|c_{\bar{n}}|& \lesssim w(B(\bar{n},1))\sup_{\mathbf{x} \in N^{-1/2}\bar{n}+N^{-1/2}[0,1)^N}\ \int_{\mathbb{R}^N}\mathcal{G}_{2/c}(\mathbf{x},\mathbf{y})|f(\mathbf{y})|\,dw(\mathbf{y}) \\&\lesssim\int_{\mathbb{R}^N}\exp\left(-c'\dz(\bar{n},\mathbf{y})^2\right)|f(\mathbf{y})|\,dw(\mathbf{y}).
\end{align*}
Finally, it is easily seen that
\begin{align*}
\sum_{\bar{n} \in \mathbb{Z}^N}\exp\left(-c'\|\bar{n}-g(\mathbf{y})\|^2\right) \leq C
\end{align*}
independently of $g \in G$ and $\mathbf{y} \in \mathbb{R}^N$ which, thanks to~\eqref{eq:equivalent_G_t}, proves that $\sum_{\bar{n} \in \mathbb{Z}^N}|c_{\bar{n}}| \lesssim \|f\|_{L^1(\mathbb{R}^N,dw)}$, so $H^{1,{{\rm{at}}}}_{{\rm{loc}},1}$ and $H^1_{{\rm{loc}},1}$ coincide.

Let $f \in H^{1,{{\rm{at}}}}_{{\rm{loc}},T}$ be such that $\supp f \subset B(\mathbf{y}_0,1)$. Then $f$ admits an atomic decomposition $f=\sum_{j=1}^{\infty}c_j a_j$ such that $\int_{\mathbb{R}^d}a_j(\mathbf{x})\,dw(\mathbf{x})=0$ if $a_j$ is associated with a ball with radius smaller than $1$. Let $\phi$ be such a function that $\phi \in C^{\infty}_{c}(\mathbb{R}^N)$, $0 \leq \phi \leq 1$, $\supp \phi \subset B(\mathbf{y}_0,2)$, $\|\nabla \phi\| \leq 2$,  and $\phi(\mathbf{x}) \equiv 1$ for $\mathbf{x} \in B(\mathbf{y}_0,1)$. Then
\begin{align*}
&f(\mathbf{x})=f(\mathbf{x})\phi(\mathbf{x})=\sum_{j=1}^{\infty}c_j\phi(\mathbf{x})a_j(\mathbf{x}).
\end{align*}
Let as consider a single summand $\phi(\mathbf{x})a_j(\mathbf{x})=\phi(\mathbf{x})a(\mathbf{x})$. There are two cases.
\\
\textbf{Case 1.} The atom $a$ is associated with a ball $B(\mathbf{y},r)$ and $r \geq 1$. If $\|\mathbf{y}-\mathbf{y}_0\|>r+2$, then the supports of $a$ and $\phi$ are disjoint, so $a(\mathbf{x})\phi(\mathbf{x}) \equiv 0$. Otherwise, $a(\mathbf{x})\phi(\mathbf{x})$ is a multiple of atom associated with $B(\mathbf{y}_0,2)$. Indeed,
\begin{align*}
\|a\phi\|_{L^{\infty}} \leq \frac{1}{w(B(\mathbf{y},r))}=\frac{w(B(\mathbf{y},4+r))}{w(B(\mathbf{y},r))}\frac{w(B(\mathbf{y}_0,2))}{w(B(\mathbf{y},r+4))}\frac{1}{w(B(\mathbf{y}_0,2))} \leq C w(B(\mathbf{y}_0,2))^{-1}.
\end{align*} 
\\
\textbf{Case 2.} The atom $a$ is associated with a ball $B(\mathbf{y},r)$ and $r<1$. Then $\int a(\mathbf{x})\,dw(\mathbf{x})=0$. Next we prove that $\phi(\mathbf{x})a(\mathbf{x})$ can we written as a sum of $H^{1,{{\rm{at}}}}_{{\rm{loc}},1}$-atoms. Obviously,  $\|\mathbf{x}-\mathbf{y}\| \leq r$ for any $\mathbf{x} \in \supp \phi f$. We have
\begin{align*}
\phi(\mathbf{x})a(\mathbf{x})=[\phi(\mathbf{x})a(\mathbf{x})-\phi(\mathbf{y})a(\mathbf{x})]+a(\mathbf{x})\phi(\mathbf{y})=\widetilde{a}(\mathbf{x})+a(\mathbf{x})\phi(\mathbf{y}).
\end{align*}
The function $a( \cdot )\phi(\mathbf{y})$ is an atom with the cancellation condition~\eqref{numitem:cancellations}. Set $\lambda=\int_{\mathbb{R}^N}\widetilde{a}(\mathbf{x})\,dw(\mathbf{x})$. It is easily seen that $|\lambda| \leq 2r$. Let $n \in \mathbb{Z}$ be such that $2^{-n} < r \leq 2^{-n+1}$. We set
\begin{align*}
&b_0(\mathbf{x})=\widetilde{a}(\mathbf{x})-\lambda w(B(\mathbf{y},2r))^{-1} \chi_{B(\mathbf{y},2r)}(\mathbf{x}),
\end{align*}
\begin{align*}
&b_j(\mathbf{x})=w(B(\mathbf{y},2^{j}r))^{-1} \chi_{B(\mathbf{y},2^{j}r)}(\mathbf{x})-w(B(\mathbf{y},2^{j+1}r))^{-1} \chi_{B(\mathbf{y},2^{j+1}r)}(\mathbf{x})\text{ for }j=1,2,\ldots,n-1, 
\end{align*}
\begin{align*}
b_n(\mathbf{x})=w(B(\mathbf{y},2^{n}r))^{-1} \chi_{B(\mathbf{y},2^{n}r)}.
\end{align*}
Obviously, $\widetilde{a}=b_0+\sum_{j=1}^{n}\lambda b_j$. Functions $b_0,b_1,\ldots,b_{n-1}$ are multiplies of atoms with the cancellation condition~\eqref{numitem:cancellations} and $b_n$ is a multiple of an atom associated with a ball with the radius greater than $1$. Clearly, $\sum_{j=0}^{n}|\lambda| \lesssim 2^{-n}n \lesssim 1$, which finishes the proof.
\end{proof}

\subsection{Characterization by Riesz transform}

\begin{definition}\normalfont
Let $T>0$. For $j=1,2,\ldots,N$, we define the \textit{local Riesz transform} $R_j^T$  by
\begin{equation}
R_{j}^{T}f(\mathbf{x})=d_k\int_{0}^{T^2}T_j {H_{t}}f(\mathbf{x})\,\frac{dt}{\sqrt{t}},
\end{equation}
where $d_k=2^{\mathbf{N}/2}\Gamma((\mathbf{N}+1)/2)\pi^{-1/2}$.
\end{definition}
We will use Theorem~\ref{teo:Riesz} it to prove its local version.
\begin{proposition}\label{propo:local_Riesz}
Let $f \in L^1(\mathbb{R}^N,dw)$ and $T>0$. Then $f \in H^{1}_{{\rm{loc}},T}$ if and only if $R_j^{T}f \in L^1(\mathbb{R}^N,dw)$ for $j=1,2,\dots,N$. Moreover, there is a constant $C>0$, independent of $f$ and $T$, such that
\begin{equation}
C^{-1}\|f\|_{H^1_{{\rm{loc}},T}} \leq \|f\|_{L^1(\mathbb{R}^N,dw)}+\sum_{j=1}^{N}\|R_j^{T}f\|_{L^1(\mathbb{R}^N,dw)} \leq C \|f\|_{H^1_{{\rm{loc}},T}}.
\end{equation}
\end{proposition}

\begin{proof}
It suffices to prove the proposition for $T=1$, then use the rescaling property~\eqref{eq:rescaling}. To prove the first inequality it is enough to show that the estimation
\begin{align*}
\|R_{j}({H_{1/2}}-I)f\|_{L^1(\mathbb{R}^N,dw)} \lesssim \|f\|_{L^1(\mathbb{R}^N,dw)}+\|{R^{1}_{j}}f\|_{L^1(\mathbb{R}^N,dw)}
\end{align*}
holds (see Proposition~\ref{propo:H_12_I_characterization}). To this end, we write
\begin{equation}
\label{eq:calculation_for_R_j}
\begin{split}
&R_j({H_{1/2}}-I)f(\mathbf{x})=d_k\int_{0}^{\infty}T_j {H_{t}}({H_{1/2}}-I)f(\mathbf{x})\,\frac{dt}{\sqrt{t}}\\&=d_k\int_{0}^{1}T_j {H_{t}}({H_{1/2}}-I)f(\mathbf{x})\,\frac{dt}{\sqrt{t}}+d_k\int_{1}^{\infty}T_j {H_{t}}({H_{1/2}}-I)f(\mathbf{x})\,\frac{dt}{\sqrt{t}}.
\end{split}
\end{equation}
Observe that
\begin{align*}
\left\|\int_{0}^{1}T_{j}{H_t}({H_{1/2}}-I)f\,\frac{dt}{\sqrt{t}}\right\|_{L^1(\mathbb{R}^N,dw)}&=\left\|\int_{0}^{1}({H_{1/2}}-I)T_{j}{H_t}f\,\frac{dt}{\sqrt{t}}\right\|_{L^1(\mathbb{R}^N,dw)} \\ &\leq \|({H_{1/2}}-I){R^{1}_{j}}f\|_{L^1(\mathbb{R}^N,dw)} \lesssim \|{R^1_{j}}f\|_{L^1(\mathbb{R}^N,dw)} .
\end{align*}
To prove $L^1(\mathbb{R}^N,dw)$-boundedness of the second integral in~\eqref{eq:calculation_for_R_j}, it is enough to show that there is $C > 0$ such that for any $\mathbf{y} \in \mathbb{R}^N$ one has

\begin{equation}
\label{eq:difference_derivative}
\begin{split}
&J(\mathbf{y})=\int_{\mathbb{R}^N}\int_{1}^{\infty}\left|T_j\left[{h_{t+1/2}}(\mathbf{x},\mathbf{y})-{h_{t}}(\mathbf{x},\mathbf{y})\right]\right|\frac{dt}{\sqrt{t}}\,dw(\mathbf{x})\\&\leq\int_{\mathbb{R}^N}\int_{1}^{\infty}\int_{t}^{t+1/2}\left|\partial_{s}T_j{h_s}(\mathbf{x},\mathbf{y})\right|\,ds\,\frac{dt}{\sqrt{t}}\,dw(\mathbf{x}) \leq C.
\end{split}
\end{equation}

Observe that by Theorem~\ref{teo:theoremGauss}~\eqref{numitem:space} with $m=1$ we have
\begin{align*}
J(\mathbf{y}) &\lesssim \int_{\mathbb{R}^N}\int_{1}^{\infty}\int_{t}^{t+1/2}s^{-2}\mathcal{G}_{s/c}(\mathbf{x},\mathbf{y})\,ds\,dt\,dw(\mathbf{x}) \\&\lesssim \int_{\mathbb{R}^N}\int_{1}^{\infty}t^{-2}\mathcal{G}_{t/c'}(\mathbf{x},\mathbf{y})\,dt\,dw(\mathbf{x}) \lesssim 1.
\end{align*}

For the converse, assume that $f \in H^{1}_{{\rm{loc}},T}$. Then $R_j({H_{1/2}}-I)f \in L^{1}(\mathbb{R}^N,dw)$. We have
\begin{align*}
R_j^{1}f&=R_j^{1}\left(f-{H_{1/2}}f\right)+R_j^{1}{H_{1/2}}f\\&=R_j(I-{H_{1/2}})f+R_{j}^{1}{H_{1/2}}f-d_k\int_{1}^{\infty}T_j{H_t}(I-{H_{1/2}})f\,\frac{dt}{\sqrt{t}}=I_1+I_2+I_3.
\end{align*}
Note that the $L^1(\mathbb{R}^N,dw)$ norm of $I_3$ is estimated by $\|f\|_{L^1(\mathbb{R}^N,dw)}$ (see~\eqref{eq:difference_derivative}). Further,  the $L^1(\mathbb{R}^N,dw)$ norm of $I_1$ is estimated by $\|f\|_{H^{1}_{{\rm{loc}},1}}$ thanks to Proposition~\ref{propo:H_12_I_characterization} and Theorem~\ref{teo:Riesz}. For $I_2$, we see that
\begin{align*}
R_j^{1}{H_{1/2}}f=d_k\int_{1/2}^{3/2}T_{j}{H_{t}}f\frac{dt}{\sqrt{t}}.
\end{align*}
By Theorem~\ref{teo:theoremGauss}~\eqref{numitem:space} with $m=0$ we have
\begin{align*}
\left|T_{j}{h_{t}}(\mathbf{x},\mathbf{y})\right| \lesssim \frac{1}{\sqrt{t}}\mathcal{G}_{t/c}(\mathbf{x},\mathbf{y}).
\end{align*}
So finally $\|R_j^{1}f\|_{L^1(\mathbb{R}^N,dw)} \lesssim \|f\|_{H^{1}_{{\rm{loc}},1}}+\|f\|_{L^1(\mathbb{R}^N,dw)}$.
\end{proof}

\section{Partition of unity}
In this section, we introduce a covering of $\mathbb{R}^N$ by balls related to $\rho$ (see~\eqref{eq:rho}). We will use these objects to characterize Hardy spaces for the Dunkl harmonic oscillator.

It is easy to prove the following two lemmas.
\begin{lemma}\label{lem:rho_property}
Suppose that $K>0$ is a given constant. There exists a constant $C=C(K)>0$ such that for every $\mathbf{x} \in \mathbb{R}^N$ and $\mathbf{y} \in \mathcal{O}(B(\mathbf{x},K\rho(\mathbf{x})))$ we have
\begin{equation}
C^{-1}\rho(\mathbf{y}) \leq \rho(\mathbf{x}) \leq C\rho(\mathbf{y}).
\end{equation}
\end{lemma}

\begin{lemma}\label{lem:partition_of_unity_invariant}
There exist collections of balls $\{B(\mathbf{y}_m,\rho(\mathbf{y}_m))\}$ and functions $\{\psi_m\}$ such that
\begin{enumerate}[(A)]
\item{the balls $\{B(\mathbf{y}_m,\rho(\mathbf{y}_m))\}$ are pairwise disjoint,}
\item{$\bigcup_{m=1}^{\infty}B(\mathbf{y}_m,2\rho(\mathbf{y}_m))=\mathbb{R}^N$,}
\item{the balls $\{B(\mathbf{y}_m,3\rho(\mathbf{y}_m))\}$ have a finite covering property (i.e. there exists a constant $M>0$ such that any point from $\mathbb{R}^N$ belongs to at most $M$ balls from $\{B(\mathbf{y}_m,3\rho(\mathbf{y}_m))\}$),}\label{numitem:finite_covering}
\item{$\{\psi_m\}$ are $C_{c}^{\infty}(\mathbb{R}^N)$ functions and $\psi_m \geq 0$,}
\item{$\supp \psi_m \subset B(\mathbf{y}_m,3\rho(\mathbf{y}_m))$ for every $m$,}
\item{$\sum_{m=1}^{\infty}\psi_m \equiv 1$,}
\item{$\|\nabla \psi_m \| \lesssim \rho(\mathbf{y}_m)^{-1}$ for every $m$.}
\end{enumerate}
\end{lemma}

\begin{lemma}\label{lem:psi_m_cancellation}
There is $c>0$ such that for every $m \in \mathbb{N}$, $t>0$, $\mathbf{x},\mathbf{y} \in \mathbb{R}^N$ such that
\begin{align*}
\mathbf{x} \in \mathcal{O}(B(\mathbf{y}_m,3\rho(\mathbf{y}_m))) \ \ \text{ or } \ \ \mathbf{y} \in \mathcal{O}(B(\mathbf{y}_m,3\rho(\mathbf{y}_m)))
\end{align*}
we have
\begin{equation}\label{eq:cancellation_first}
|\psi_m(\mathbf{x})-\psi_m(\mathbf{y})|\left(\frac{\|\mathbf{x}-\mathbf{y}\|}{\sqrt{t}}+1\right)^{-1}\mathcal{G}_{t/c}(\mathbf{x},\mathbf{y}) \lesssim \frac{\sqrt{t}}{\rho(\mathbf{y})} \mathcal{G}_{t/c}(\mathbf{x},\mathbf{y}).
\end{equation}
In particular,
\begin{equation}\label{eq:cancellation_second}
|\psi_m(\mathbf{x})-\psi_m(\mathbf{y})|{h_t}(\mathbf{x},\mathbf{y}) \lesssim \frac{\sqrt{t}}{\rho(\mathbf{y})} \mathcal{G}_{t/c}(\mathbf{x},\mathbf{y}).
\end{equation}
\end{lemma}

\begin{proof}
\textbf{Case 1.} $\dz(\mathbf{x},\mathbf{y}) \geq \rho(\mathbf{y})$. Then $\|\mathbf{x}-\mathbf{y}\| \geq \rho(\mathbf{y})$ and~\eqref{eq:cancellation_first} easily follows.
\\
\textbf{Case 2.} $\dz(\mathbf{x},\mathbf{y}) < \rho(\mathbf{y})$. By Lemma~\ref{lem:rho_property} we have $\rho(\mathbf{y}) \lesssim \rho(\mathbf{x}) \lesssim \rho(\mathbf{y}_m)$, so by the mean value theorem,
\begin{align*}
|\psi_m(\mathbf{x})-\psi_m(\mathbf{y})|\lesssim \frac{\|\mathbf{x}-\mathbf{y}\|}{\rho(\mathbf{y}_m)}.
\end{align*}
If $\|\mathbf{x}-\mathbf{y}\|=0$, then the claim is obvious. If $\|\mathbf{x}-\mathbf{y}\| \neq 0$, then
\begin{align*}
|\psi_m(\mathbf{x})-\psi_m(\mathbf{y})|\left(\frac{\|\mathbf{x}-\mathbf{y}\|}{\sqrt{t}}+1\right)^{-1}\mathcal{G}_{t/c}(\mathbf{x},\mathbf{y}) \lesssim \frac{\sqrt{t}}{\|\mathbf{x}-\mathbf{y}\|} \frac{\|\mathbf{x}-\mathbf{y}\|}{\rho(\mathbf{y}_m)}\mathcal{G}_{t/c}(\mathbf{x},\mathbf{y}) \lesssim \frac{\sqrt{t}}{\rho(\mathbf{y})}\mathcal{G}_{t/c}(\mathbf{x},\mathbf{y}).
\end{align*}
The estimation~\eqref{eq:cancellation_second} is a direct consequence of~\eqref{eq:cancellation_first} and Theorem~\ref{teo:theoremGauss}~\eqref{numitem:Gauss}.
\end{proof}

\section{Hardy spaces for the Dunkl harmonic oscillator\\ proofs of atomic decomposition and Riesz transform characterization}

\subsection{Proof of  Theorem~\ref{teo:atomic_Dunkl_Hermite}}
The kernels $k_t(\mathbf{x},\mathbf{y})$ and ${h_{t_1}}(\mathbf{x},\mathbf{y})$  of the  Dunkl and Dunkl--Hermite semigroups  have locally  similar behavior for small $t$.  On the other hand  the kernel $k_t(\mathbf x,\mathbf y)$ has faster decay for large $t$ (see \eqref{eq:connection_Dunkl_Dunkl_Hermite}). Analysis of the kernels combined with characterizations of local Hardy spaces allow us to prove Theorem~\ref{teo:atomic_Dunkl_Hermite}.

In the remaining part of the paper $t,t_1>0$ are always related by
 $$ t_1=\frac{1}{2}\sinh(2t).$$
 The lemma below is a list of basic properties of $t$ and $t_1$, the proof is omitted.
\begin{lemma}\label{lem:t_t_1_estimations} For $t>0$ we have
\begin{enumerate}[(A)]
\item{$t \leq t_1$ for $t>0$ and $t_1 \leq 2t$ for $0<t<1$,}\label{numitem:t<t_1}
\item{$0 \leq t_1-t \lesssim t^3$ for $0<t<1$,}\label{numitem:t^3}
\item{$\frac{1}{4t_1}-\frac{1}{2}\coth(2t)+\frac{t}{4}<0$ for $0<t<1$ and $\left|\frac{1}{4t_1}-\frac{1}{2}\coth(2t)\right| \lesssim t$ for $t>0$,}\label{numitem:coth_t}
\item{$\frac{1}{4t_1}-\frac{1}{2}\coth(2t) \leq -\frac{1}{4}$ for $t \geq 1$,}\label{numitem:coth>C}
\item{$t_1^{-1} \lesssim e^{-2t}$ for $t \geq 1$.}\label{numitem:ext-t}

\end{enumerate}
\end{lemma}

\begin{lemma}\label{lem:H_t_H_t_1_difference}
There exist constants $C,c>0$ such that for $0<t<1$ we have
\begin{enumerate}[(A)]
\item{$|{h_t}(\mathbf{x},\mathbf{y})-{h_{t_1}}(\mathbf{x},\mathbf{y})| \leq Ct^2 \mathcal{G}_{t/c}(\mathbf{x},\mathbf{y})$,}\label{numitem:H_t_H_t1_difference}
\item{$|{h_{t_1}}(\mathbf{x},\mathbf{y})-k_t(\mathbf{x},\mathbf{y})| \leq Ct(\|\mathbf{x}\|^2+\|\mathbf{y}\|^2)\min\left\{\mathcal{G}_{t/c}(\mathbf{x},\mathbf{y}),{h_{t_1}}(\mathbf{x},\mathbf{y})\right\}$.}\label{numitem:Dunkl_Dunkl_Hermite_difference}
\end{enumerate}
\end{lemma}

\begin{proof}
For~\eqref{numitem:H_t_H_t1_difference}, we have
\begin{align*}
&\left|{h_t}(\mathbf{x},\mathbf{y})-{h_{t_1}}(\mathbf{x},\mathbf{y})\right|=\left|\int_{t_1}^{t}\partial_s {h_s}(\mathbf{x},\mathbf{y}) \,ds\right| \leq |t-t_1|\sup_{s \in [t,t_1]}\left|\partial_s h_s(\mathbf{x},\mathbf{y})\right|,
\end{align*}
so using Theorem~\ref{teo:theoremGauss}~\eqref{numitem:time} and Lemma~\ref{lem:t_t_1_estimations}~\eqref{numitem:t^3} we are done.

For~\eqref{numitem:Dunkl_Dunkl_Hermite_difference}, we recall that
\begin{align*}
k_t(\mathbf{x},\mathbf{y})-{h_{t_1}}(\mathbf{x},\mathbf{y})=\left[\exp\left(\left(\frac{1}{4t_1}-\frac{1}{2}\coth(2t)\right)(\|\mathbf{x}\|^2+\|\mathbf{y}\|^2)\right)-1\right]{h_{t_1}}(\mathbf{x},\mathbf{y}).
\end{align*}
The proof is completed by Lemma~\ref{lem:t_t_1_estimations}~\eqref{numitem:coth_t} .
\end{proof}

\begin{proposition}\label{propo:H_t_H_t_1_difference}
There exists a constant $C>0$ such that for any $\mathbf{y}_0 \in \mathbb{R}^N$ and for any $f \in L^{1}(B(\mathbf{y}_0,3\rho(\mathbf{y}_0)),dw)$ we have
\begin{equation}
\left\|\sup_{0<t \leq \rho(\mathbf{y}_0)^2}\left|{H_t}f(\cdot)-K_tf(\cdot)\right|\right\|_{L^1(\mathbb{R}^N,dw)} \leq C\|f\|_{L^1(\mathbb{R}^N,dw)}.
\end{equation}
\end{proposition}

\begin{proof}
It is enough to prove that
\begin{align*}
& \int_{\mathbb{R}^N} \sup_{0<t \leq \rho(\mathbf{y}_0)^2} |{h_t}(\mathbf{x},\mathbf{y})-k_t(\mathbf{x},\mathbf{y})|\,dw(\mathbf{x}) \leq C
\end{align*}
for $\mathbf{y} \in B(\mathbf{y}_0,3\rho(\mathbf{y}_0))$ with $C>0$ which does not depend on $\mathbf{y}_0$ and $\mathbf{y}$. We split the integral into two parts: over $(\mathcal{O}(B(\mathbf{y}_0,6\rho(\mathbf{y}_0))))^{c}$ and over $\mathcal{O}(B(\mathbf{y}_0,6\rho(\mathbf{y}_0)))$. For the first one, it suffices to prove the boundedness of the integrals of ${h_t}(\mathbf{x},\mathbf{y})$ and $k_t(\mathbf{x},\mathbf{y})$ separately. Each of them is bounded by $\mathcal{G}_{t/c}(\mathbf{x},\mathbf{y})$. Note that if $\mathbf{x} \not\in \mathcal{O}(B(\mathbf{y}_0,6\rho(\mathbf{y}_0)))$, then $\dz(\mathbf{x},\mathbf{y})\geq \rho(\mathbf{y}_0)$. Therefore, by Theorem~\ref{teo:theoremGauss}~\eqref{numitem:Gauss} and Lemma~\ref{lem:substitute_t_by_dz}, we are reduced to estimate
\begin{align*}
\int_{\mathbb{R}^N}\frac{1}{w(B(\mathbf{y}_0,\rho(\mathbf{y}_0)))}\exp\left(-c\rho(\mathbf{y}_0)^{-2}\dz(\mathbf{x},\mathbf{y})^2\right)\,dw(\mathbf{x}).
\end{align*}
The integral above is bounded by constant, which is independent of $\mathbf{y}$ and $\mathbf{y}_0$ (see~\eqref{eq:integral_of_kernel}).
For the integral over $\mathcal{O}(B(\mathbf{y}_0,6\rho(\mathbf{y}_0)))$, we have $\|\mathbf{x}\|^2+\|\mathbf{y}\|^2 \lesssim \rho(\mathbf{y}_0)^{-2}$, so, by Lemma~\ref{lem:H_t_H_t_1_difference}, it is enough to prove that
\begin{align*}
&\int_{\mathcal{O}(B(\mathbf{y}_0,6\rho(\mathbf{y}_0)))}\sup_{0<t \leq \rho(\mathbf{y}_0)^2} t\rho^{-2}(\mathbf{y}_0)\mathcal{G}_{t/c}(\mathbf{x},\mathbf{y})\,dw(\mathbf{x}) \leq C
\end{align*}
for $\mathbf{y} \in \mathcal{O}(B(\mathbf{y}_0,3\rho(\mathbf{y}_0)))$, with $C>0$ which does not depend on $\mathbf{y}_0$ and $\mathbf{y}$. It is easy to see that $\mathcal{O}(B(\mathbf{y}_0,6\rho(\mathbf{y}_0))) \subset \mathcal{O}(B(\mathbf{y},12\rho(\mathbf{y}_0)))$ and
\begin{align*}
&\sup_{0<t \leq \rho(\mathbf{y}_0)^2} t\rho^{-2}(\mathbf{y}_0)\mathcal{G}_{t/c}(\mathbf{x},\mathbf{y})  \leq \sup_{0<t \leq \rho(\mathbf{y}_0)^2} t^{1/2}\rho^{-1}(\mathbf{y}_0)\mathcal{G}_{t/c}(\mathbf{x},\mathbf{y}).
\end{align*}
Using Lemma~\ref{lem:substitute_t_by_dz} with $\alpha=-1/2$, we see that it is enough to estimate
\begin{align*}
\int_{\mathcal{O}(B(\mathbf{y},12\rho(\mathbf{y}_0)))}\dz(\mathbf{x},\mathbf{y})\rho(\mathbf{y}_0)^{-1}\frac{1}{w(B(\mathbf{y},\dz(\mathbf{x},\mathbf{y})))}\,dw(\mathbf{x}).
\end{align*}
To this end we apply Lemma~\ref{lem:summation} with $\alpha=1/2$.
\end{proof}

\begin{lemma}\label{lem:atom_in_H1_Dunkl_Hermite}
For any $A>1$ there is $C>0$ such that for any function $a(\mathbf{x})$ such that $\supp a \subset B(\mathbf{p},r)$ for some $\mathbf{p} \in \mathbb{R}^N$, $\sup_{\mathbf{x} \in \mathbb{R}^N}|a(\mathbf{x})| \leq w(B(\mathbf{p},r))^{-1}$ and $\rho(\mathbf{p}) \leq r \leq A\rho(\mathbf{p})$ we have
\begin{equation}
\|a\|_{\mathcal{H}^{1}} \leq C.
\end{equation}
\end{lemma}

\begin{proof}
Thanks to  Lemma~\ref{lem:atom_in_H1}, the inequality $k_t(\mathbf{x},\mathbf{y}) \leq {h_{t_1}}(\mathbf{x},\mathbf{y})$ (see~\eqref{eq:connection_Dunkl_Dunkl_Hermite}), and $t \sim t_1$ for $t \leq 1$ (see Lemma~\ref{lem:t_t_1_estimations}~\eqref{numitem:t<t_1}), it is enough to prove
\begin{align*}
\int_{\mathbb{R}^N}\sup_{t > r^2}k_t(\mathbf{x},\mathbf{y})\,dw(\mathbf{x})&\leq\int_{\mathbb{R}^N}\sup_{t \geq 1}k_t(\mathbf{x},\mathbf{y})\,dw(\mathbf{x})+\int_{\mathbb{R}^N}\sup_{r^2<t<1}k_t(\mathbf{x},\mathbf{y})\,dw(\mathbf{x})\\&=J_1(\mathbf{y})+J_2(\mathbf{y}) \leq C,
\end{align*}
with $C>0$ which is independent of $r$, $\mathbf{p}$, and $\mathbf{y} \in B(\mathbf{p},r)$.  To estimate $J_1(\mathbf{y})$, we use Lemma~\ref{lem:t_t_1_estimations}~\eqref{numitem:coth>C}. Indeed, by~\eqref{eq:connection_Dunkl_Dunkl_Hermite} and Theorem~\ref{teo:theoremGauss}~\eqref{numitem:Gauss},
\begin{align*}
J_1(\mathbf{y}) \lesssim \int_{\mathbb{R}^N}\frac{1}{w(B(\mathbf{x},1))}\exp(-\|\mathbf{x}\|^2/4)\,dw(\mathbf{x}) \leq C.
\end{align*}
For the purpose of estimating $J_2(\mathbf{y})$, split the integral into integrals over $\mathcal{O}(B(\mathbf{y},r))$ and $(\mathcal{O}(B(\mathbf{y},r)))^{c}$. The first one is easy to estimate thanks to $k_t(\mathbf{x},\mathbf{y}) \lesssim w(B(\mathbf{y},\sqrt{t}))^{-1}$. For the second one, we use~\eqref{eq:connection_Dunkl_Dunkl_Hermite}, Theorem~\ref{teo:theoremGauss}~\eqref{numitem:Gauss}, and Lemma~\ref{lem:t_t_1_estimations}~\eqref{numitem:coth_t} and~\eqref{numitem:t<t_1}. We have
\begin{align*}
k_t(\mathbf{x},\mathbf{y}) \lesssim \exp\left(-\frac{t}{4}(\|\mathbf{x}\|^2+\|\mathbf{y}\|^2)\right)\exp\left(\frac{-c\dz(\mathbf{x},\mathbf{y})^2}{2t}\right)\mathcal{G}_{2t/c}(\mathbf{x},\mathbf{y}) .
\end{align*}
Since $\rho(\mathbf{y}) \lesssim \rho(\mathbf{p})$ and $\rho(\mathbf{p}) \leq r$, we have (see Lemma~\ref{lem:rho_property})
\begin{align*}
\exp\left(-\frac{t(\|\mathbf{x}\|^2+\|\mathbf{y}\|^2)}{4}\right)\exp\left(\frac{-c\dz(\mathbf{x},\mathbf{y})^2}{2t}\right) &\lesssim \frac{1}{t(\|\mathbf{x}\|^2+\|\mathbf{y}\|^2+1)}\frac{t}{\dz(\mathbf{x},\mathbf{y})^2} \\&\lesssim \frac{r^{2}}{\dz(\mathbf{x},\mathbf{y})^2}.
\end{align*}
The rest of the proof follows by  Lemma~\ref{lem:substitute_t_by_dz} with $\alpha=0$ and Lemma~\ref{lem:summation} with $\alpha=1$.
\end{proof}

As the consequence of the proof of Lemma~\ref{lem:atom_in_H1_Dunkl_Hermite} we obtain the following corollary.

\begin{corollary}\label{coro:Dunkl_Hermite_for_large_t}
There exists a constant $C>0$ such that for any $\mathbf{y}_0 \in \mathbb{R}^N$ and for any $f \in L^{1}(B(\mathbf{y}_0,\rho(\mathbf{y}_0)),dw)$ we have
\begin{equation}
\|\sup_{t> \rho(\mathbf{y}_0)^2}K_tf(\cdot)\|_{L^1(\mathbb{R}^N,dw)} \leq C\|f\|_{L^1(\mathbb{R}^N,dw)}.
\end{equation}
\end{corollary}

\begin{proof}[Proof of Theorem~\ref{teo:atomic_Dunkl_Hermite}]
In order to prove $\mathcal{H}^{1,\rm{at}} \subseteq \mathcal{H}^1$ it is enough to show that there exists a constant $C>0$ such that $\|a\|_{\mathcal{H}^1} \leq C$ for any atom $a$ for $\mathcal{H}^{1,{{\rm{at}}}}$. Fix an atom $a$ for $\mathcal{H}^{1,\rm{at}}$. Let $B(\mathbf{x},r)$ be as in Definition~\ref{def:atomic_Dunkl_Hermite}. We have two cases: $r<\rho(\mathbf{x})$ or $\rho(\mathbf{x}) \leq r \leq A\rho(\mathbf{x})$.
In the second case, we use Lemma~\ref{lem:atom_in_H1_Dunkl_Hermite}. In the first case, the required inequality is a consequence of Proposition~\ref{propo:H_t_H_t_1_difference}, Corollary~\ref{coro:Dunkl_Hermite_for_large_t}, and Proposition~\ref{propo:Goldberg}.

For the opposite inclusion, take $f \in \mathcal{H}^1$. Let $\{\psi_m\}$ be the partition of unity from Lemma~\ref{lem:partition_of_unity_invariant} associated with the collection of balls  $\{B(\mathbf{y}_m,\rho(\mathbf{y}_m))\}$. Then $f=\sum_{m=1}^{\infty}f\psi_m$. We are going to prove that there is a constant $C>0$ independent of $f$ such that
\begin{equation}
\label{eq:main_goal}
\sum_{m=1}^{\infty}\|f \psi_m \|_{H^{1}_{{\rm loc},\rho(\mathbf{y}_m)}} \leq C\|f\|_{\mathcal{H}^{1}}.
\end{equation}
To this end we write

\begin{align*}
&\sum_{m=1}^{\infty}\|f \psi_m \|_{H^{1}_{{\rm{loc}},\rho(\mathbf{y}_m)}} = \sum_{m=1}^{\infty}\left\|\sup_{0<t \leq \rho(\mathbf{y}_m)^2}\left|{H_t}(f\psi_m)\right|\right\|_{L^1(\mathbb{R}^N,dw)}\\&\leq
\sum_{m=1}^{\infty}\left\|\sup_{0<t \leq \rho(\mathbf{y}_m)^2}|({H_t}-K_t)(f\psi_m)|\right\|_{L^1(\mathbb{R}^N,dw)}\\&+
\sum_{m=1}^{\infty}\left\|\sup_{0<t \leq \rho(\mathbf{y}_m)^2}\left|K_t(f\psi_m)(\cdot)-\psi_m(\cdot)K_t f(\cdot)\right|\right\|_{L^1(\mathbb{R}^N,dw)}\\&+
\sum_{m=1}^{\infty}\left\|\psi_m\sup_{0<t \leq \rho(\mathbf{y}_m)^2}|K_tf|\right\|_{L^1(\mathbb{R}^N,dw)}=I_1+I_2+I_3.
\end{align*}
Obviously, $I_3 \leq \|f\|_{\mathcal{H}^1}$. By Lemma~\ref{propo:H_t_H_t_1_difference}, $I_1 \lesssim \|f\|_{L^1(\mathbb{R}^N,dw)} \leq \|f\|_{\mathcal{H}^1}$.

Now we are going to estimate $I_2$ by $\|f\|_{L^1(\mathbb{R}^N,dw)}$. For this purpose it is enough to show that
\begin{align*}
\sup_{\mathbf{y} \in \mathbb{R}^N}\sum_{m=1}^{\infty}\int \sup_{0<t \leq \rho(\mathbf{y}_m)^2} |\psi_m(\mathbf{y})-\psi_m(\mathbf{x})|k_t(\mathbf{x},\mathbf{y})\,dw(\mathbf{x})<\infty.
\end{align*}
Fix $\mathbf{y} \in \mathbb{R}^N$ and consider a single summand. Split the integral above into two parts: over $\mathcal{O}((B(\mathbf{y}_m,6\rho(\mathbf{y}_m)))^{c}$ and over $\mathcal{O}(B(\mathbf{y}_m,6\rho(\mathbf{y}_m)))$. The first integral is zero if $\mathbf{y} \not\in B(\mathbf{y}_m,3\rho(\mathbf{y}_m))$. By Lemma~\ref{lem:partition_of_unity_invariant}~\eqref{numitem:finite_covering} there are at most $M$ indices $m$ for which we have $\mathbf{y} \in B(\mathbf{y}_m,3\rho(\mathbf{y}_m))$ and $M$ is independent of $\mathbf{y}$. Assume that $\mathbf{y} \in B(\mathbf{y}_m,3\rho(\mathbf{y}_m))$, by Lemma~\ref{lem:substitute_t_by_dz}, we have
\begin{align*}
&\int_{(\mathcal{O}(B(\mathbf{y}_m,6\rho(\mathbf{y}_m))))^{c}}\sup_{0<t \leq \rho^2(\mathbf{y}_m)} |\psi_m(\mathbf{y})-\psi_m(\mathbf{x})|k_t(\mathbf{x},\mathbf{y})\,dw(\mathbf{x}) \\& \lesssim \int_{(\mathcal{O}(B(\mathbf{y}_m,6\rho(\mathbf{y}_m))))^{c}}\frac{1}{w(B(\mathbf{x},\dz(\mathbf{x},\mathbf{y})))}\exp\left(\frac{-c\dz(\mathbf{x},\mathbf{y})^2}{\rho(\mathbf{y}_m)^2}\right)\,dw(\mathbf{x}) \\&\lesssim \int_{\mathbb{R}^N}\frac{1}{w(B(\mathbf{x},\rho(\mathbf{y}_m)))}\exp\left(\frac{-c\dz(\mathbf{x},\mathbf{y})^2}{\rho(\mathbf{y}_m)^2}\right)\,dw(\mathbf{x}) \leq C,
\end{align*}
with $C>0$ which is independent of $\mathbf{y}_m$ (see~\eqref{eq:integral_of_kernel}).

Therefore, it remains to show that
\begin{equation}\label{eq:partition_integrals}
\sum_{m=1}^{\infty}\int_{\mathcal{O}(B(\mathbf{y}_m,6\rho(\mathbf{y}_m))))} \sup_{0<t<\rho(\mathbf{y}_m)^2} |\psi_m(\mathbf{y})-\psi_m(\mathbf{x})|k_t(\mathbf{x},\mathbf{y})\,dw(\mathbf{x}) \lesssim C
\end{equation}
with $C>0$ which is independent of $\mathbf{y} \in \mathbb{R}^N$. Split a single integral into two parts: the integral over $\mathcal{O}(B(\mathbf{y}_m,6\rho(\mathbf{y}_m))))\cap (\mathcal{O}(B(\mathbf{y},2\rho(\mathbf{y}))))^{c}$ and over $\mathcal{O}(B(\mathbf{y}_m,6\rho(\mathbf{y}_m)))) \cap \mathcal{O}(B(\mathbf{y},2\rho(\mathbf{y})))$.

We consider each term separately. For the first one, by ~\eqref{eq:connection_Dunkl_Dunkl_Hermite}, Theorem~\ref{teo:theoremGauss}~\eqref{numitem:Gauss}, Lemma~\ref{lem:t_t_1_estimations}~\eqref{numitem:t<t_1}, Lemma~\ref{lem:substitute_t_by_dz}, and Lemma~\ref{lem:rho_property} for $0<t<\rho(\mathbf{y}_m)^2$ we have
\begin{align*}
&|\psi_m(\mathbf{y})-\psi_m(\mathbf{x})|k_t(\mathbf{x},\mathbf{y}) \lesssim {h_{t_1}}(\mathbf{x},\mathbf{y}) \lesssim \min\left\{1,\frac{t_1}{\|\mathbf{x}-\mathbf{y}\|^2}\right\}\mathcal{G}_{t_1/c}(\mathbf{x},\mathbf{y}) \\&\lesssim \min\left\{1,\frac{\rho(\mathbf{y}_m)^2}{\|\mathbf{x}-\mathbf{y}\|^2}\right\} w(B(\mathbf{y},\dz(\mathbf{x},\mathbf{y})))^{-1} \lesssim \min\left\{1,\frac{\rho(\mathbf{x})^2}{\|\mathbf{x}-\mathbf{y}\|^2}\right\} w(B(\mathbf{y},\dz(\mathbf{x},\mathbf{y})))^{-1}.
\end{align*}
If $\rho(\mathbf{x}) \leq 2 \rho(\mathbf{y})$, then
\begin{align*}
\min\left\{1,\frac{\rho(\mathbf{x})^2}{\|\mathbf{x}-\mathbf{y}\|^2}\right\} & \lesssim \frac{\rho(\mathbf{y})}{\|\mathbf{x}-\mathbf{y}\|} \leq \frac{\rho(\mathbf{y})}{\dz(\mathbf{x},\mathbf{y})}.
\end{align*}
If $\rho(\mathbf{x}) \geq 2 \rho(\mathbf{y})$ it is easy to check that $\frac{1}{\|\mathbf{x}-\mathbf{y}\|} \lesssim \rho(\mathbf{y})$, so
\begin{align*}
\frac{\rho(\mathbf{x})^2}{\|\mathbf{x}-\mathbf{y}\|^2} \leq \frac{1}{\|\mathbf{x}-\mathbf{y}\|^2}\lesssim \frac{\rho(\mathbf{y})}{\|\mathbf{x}-\mathbf{y}\|} \leq \frac{\rho(\mathbf{y})}{\dz(\mathbf{x},\mathbf{y})}.
\end{align*}
Finally, 
\begin{align*}
|\psi_m(\mathbf{y})-\psi_m(\mathbf{x})|k_t(\mathbf{x},\mathbf{y}) \lesssim\frac{\rho(\mathbf{y})}{\dz(\mathbf{x},\mathbf{y})} w(B(\mathbf{y},\dz(\mathbf{x},\mathbf{y})))^{-1},
\end{align*}
and, by the finite covering property,
\begin{align*}
&\sum_{m=1}^{\infty}\int_{\mathcal{O}(B(\mathbf{y}_m,3\rho(\mathbf{y}_m))))\cap (\mathcal{O}(B(\mathbf{y},2\rho(\mathbf{y})))^{c}}\sup_{0<t<\rho(\mathbf{y}_m)^2} |\psi_m(\mathbf{y})-\psi_m(\mathbf{x})|k_t(\mathbf{x},\mathbf{y})\,dw(\mathbf{x})\\& \lesssim
\sum_{m=1}^{\infty}\int_{\mathcal{O}(B(\mathbf{y}_m,3\rho(\mathbf{y}_m))))\cap (\mathcal{O}(B(\mathbf{y},2\rho(\mathbf{y})))^{c}}\frac{\rho(\mathbf{y})}{\dz(\mathbf{x},\mathbf{y})} w(B(\mathbf{y},\dz(\mathbf{x},\mathbf{y})))^{-1}\,dw(\mathbf{x})\\&\lesssim
\int_{(\mathcal{O}(B(\mathbf{y},2\rho(\mathbf{y})))^{c}}\frac{\rho(\mathbf{y})}{\dz(\mathbf{x},\mathbf{y})} w(B(\mathbf{y},\dz(\mathbf{x},\mathbf{y})))^{-1}\,dw(\mathbf{x}).
\end{align*}
The last integral is bounded by constant by Lemma~\ref{lem:summation}~\eqref{numitem:summation_outside} with $\alpha=1/2$.
\\
For the sum of the integrals over $\mathcal{O}(B(\mathbf{y}_m,6\rho(\mathbf{y}_m))))\cap\mathcal{O}(B(\mathbf{y},2\rho(\mathbf{y})))$ in ~\eqref{eq:partition_integrals}, by~\eqref{eq:connection_Dunkl_Dunkl_Hermite} and Lemma~\ref{lem:psi_m_cancellation}, we have
\begin{align*}
|\psi_m(\mathbf{y})-\psi_m(\mathbf{x})|k_t(\mathbf{x},\mathbf{y}) \lesssim \frac{\sqrt{t}}{\rho(\mathbf{y})}\mathcal{G}_{2t_1/c}(\mathbf{x},\mathbf{y}).
\end{align*}
Thus, we apply Lemma~\ref{lem:substitute_t_by_dz} with $\alpha=-1/2$ and, by the finite covering property, reduce the problem to an estimation of
\begin{align*}
\int_{\mathcal{O}(B(\mathbf{y},2\rho(\mathbf{y})))}\dz(\mathbf{x},\mathbf{y})\rho(\mathbf{y})^{-1}\frac{1}{w(B(\mathbf{y},\dz(\mathbf{x},\mathbf{y})))}\,dw(\mathbf{x}).
\end{align*}
The integral is bounded by a constant independently of $\mathbf{y}$ by Lemma~\ref{lem:summation} with $\alpha=1/2$. Thus~\eqref{eq:main_goal} is established.

We have obtained that each $f\psi_m$ belongs to  $H_{{\rm{loc}}, \rho(\mathbf{y}_m)}^{1}$.  So, by Proposition~\ref{propo:Goldberg},  $f\psi_m$ admits an atomic decomposition $f(\mathbf{x})\psi_m(\mathbf{x})=\sum_{j=1}^{\infty}c_j^{m}a_j^m(\mathbf{x})$, where $a_{j}^{m}$ are $H^{1,{{\rm{at}}}}_{{\rm{loc}}, \rho(\mathbf{y}_m)}$ atoms such that $\supp a_j^m \subset B(\mathbf{y}_m,12\rho(\mathbf{y}_m))$ and
\begin{equation}\label{eq:partial_atomic}
\sum_{j=1}^{\infty}|c_j^{m}| \lesssim \|f\psi_m\|_{\mathcal{H}^{1}_{{\rm{loc}},\rho(\mathbf{y}_m)}}.
\end{equation}
Observe that by Lemma~\ref{lem:rho_property} with $K=12$, $a_{j}^m$ are $\mathcal{H}^{1,{{\rm{at}}}}$ atoms if we take $A>1$ large enough. Finally, by~\eqref{eq:partial_atomic} and~\eqref{eq:main_goal}, we obtain
\begin{align*}
f(\mathbf{x})=\sum_{m=1}^{\infty}f(\mathbf{x})\psi_m(\mathbf{x})=\sum_{m=1}^{\infty}\sum_{j=1}^{\infty}c_j^{m}a_j^m(\mathbf{x}) \qquad \text{and} \qquad \sum_{m=1}^{\infty}\sum_{j=1}^{\infty}|c_{j}^{m}| \lesssim \|f\|_{\mathcal{H}^1}.
\end{align*}
\end{proof}

\subsection{Proof of Theorem~\ref{teo:Riesz_Hermite}}
Using the fact that $T_{j,\mathbf{x}}E(\mathbf{x},\mathbf{y})=y_j E(\mathbf{x},\mathbf{y})$  together with~\eqref{eq:connection_Dunkl_Dunkl_Hermite}, we obtain

\begin{equation}
\label{eq:formula_Dj}
T_{j,\mathbf{x}}{h_t}(\mathbf{x},\mathbf{y})=\frac {y_j-x_j}{2t}{h_t}(\mathbf{x},\mathbf{y}),
\end{equation}

\begin{equation}
\label{eq:formula_T_jH_t}
\begin{split}
T_{j,\mathbf{x}} k_t(\mathbf{x},\mathbf{y})=&\frac{y_j-x_j}{2t_1}k_t(\mathbf{x},\mathbf{y})+2x_j\left(\frac{1}{4t_1}-\frac{1}{2}\coth(2t)\right)k_t(\mathbf{x},\mathbf{y})\\=&\exp\left(\left(\frac{1}{4t_1}-\frac{1}{2}\coth(2t)\right)(\|\mathbf{x}\|^2+\|\mathbf{y}\|^2)\right)T_{j,\mathbf{x}}{h_{t_1}}(\mathbf{x},\mathbf{y})\\&+2x_j\left(\frac{1}{4t_1}-\frac{1}{2}\coth(2t)\right)k_t(\mathbf{x},\mathbf{y}),
\end{split}
\end{equation}
where here and subsequently $T_{j,\mathbf{x}}$ denotes the action of $T_j$ with respect to the variable $\mathbf{x}$.
\begin{lemma}\label{lem:estimation_T_jH_t}
There are constants $C,c>0$ such that for any $\mathbf{x},\mathbf{y} \in \mathbb{R}^N$, $t>0$ we have
\begin{align*}
\left|\frac{1}{\sqrt{t}}T_{j,\mathbf{x}} k_t(\mathbf{x},\mathbf{y})\right| \leq& C\left(1+\frac{1}{t}\right)\exp\left(\frac{\frac{1}{4t_1}-\frac{1}{2}\coth(2t)}{2}\left(\|\mathbf{x}\|^2+\|\mathbf{y}\|^2\right)\right)\times \\&\left(1+\frac{\|\mathbf{x}-\mathbf{y}\|}{\sqrt{t_1}}\right)^{-1}\mathcal{G}_{t_1/c}(\mathbf{x},\mathbf{y})
\end{align*}
for any $j=1,2,\ldots,N$.
\end{lemma}
\begin{proof}
By~\eqref{eq:formula_T_jH_t} we have
\begin{align*}
\left|\frac{1}{\sqrt{t}}T_{j,\mathbf{x}} k_t(\mathbf{x},\mathbf{y})\right| &\leq \frac{1}{\sqrt{t}}\left|\frac{y_j-x_j}{2t_1}k_t(\mathbf{x},\mathbf{y})\right|+\frac{1}{\sqrt{t}}\left|2x_j\left(\frac{1}{4t_1}-\frac{1}{2}\coth(2t)\right)k_t(\mathbf{x},\mathbf{y})\right|\\&=S_1+S_2.
\end{align*}
For $S_1$, note that by Theorem~\ref{teo:theoremGauss}~\eqref{numitem:Gauss} and Lemma~\ref{lem:t_t_1_estimations}~\eqref{numitem:t<t_1} we have
\begin{align*}
&S_1\exp\left(-\left(\frac{1}{4t_1}-\frac{1}{2}\coth(2t)\right)\left(\|\mathbf{x}\|^2+\|\mathbf{y}\|^2\right)\right) = \frac{1}{\sqrt{t}} \left|\frac{y_j-x_j}{2t_1}{h_{t_1}}(\mathbf{x},\mathbf{y})\right|\\&\lesssim \left(1+\frac{\|\mathbf{x}-\mathbf{y}\|}{\sqrt{t_1}}\right)^{-2}\mathcal{G}_{t_1/c}(\mathbf{x},\mathbf{y})\frac{1}{\sqrt{t}} \left|\frac{y_j-x_j}{2t_1}\right| \lesssim \frac{1}{\sqrt{t_1 t}}  \left(1+\frac{\|\mathbf{x}-\mathbf{y}\|}{\sqrt{t_1}}\right)^{-1}\mathcal{G}_{t_1/c}(\mathbf{x},\mathbf{y})\\&\lesssim \left(1+\frac{1}{t}\right) \left(1+\frac{\|\mathbf{x}-\mathbf{y}\|}{\sqrt{t_1}}\right)^{-1}\mathcal{G}_{t_1/c}(\mathbf{x},\mathbf{y}).
\end{align*}
The estimation for $S_2$ is an easy consequence of~\eqref{eq:connection_Dunkl_Dunkl_Hermite} and Theorem~\ref{teo:theoremGauss}~\eqref{numitem:Gauss}.
\end{proof}

\begin{lemma}\label{lem:Riesz_H_t_H_t1}
Let $A>0$. There exist constants $C,c>0$ such that for any $\mathbf{x}, \mathbf{y}_0 \in \mathbb{R}^N$, $j=1,2,\ldots,N$, $\mathbf{y} \in B(\mathbf{y}_0,A\rho(\mathbf{y}_0))$, and $0<t<1$ the following inequality holds
\begin{equation}
\left|\frac{1}{\sqrt{t}}\left(T_{j,\mathbf{x}} k_t(\mathbf{x},\mathbf{y})-T_{j,\mathbf{x}}{h_t}(\mathbf{x},\mathbf{y})\right)\right| \leq C\rho(\mathbf{y}_0)^{-2}\mathcal{G}_{2t/c}(\mathbf{x},\mathbf{y}).
\end{equation}
\end{lemma}

\begin{proof}
By~\eqref{eq:formula_Dj} and~\eqref{eq:formula_T_jH_t} we have
\begin{align*}
&\frac{1}{\sqrt{t}}\left(T_{j,\mathbf{x}} k_t(\mathbf{x},\mathbf{y})-T_{j,\mathbf{x}}{h_t}(\mathbf{x},\mathbf{y})\right)=\frac{1}{\sqrt{t}}\left[\frac{y_j-x_j}{2t_1}\left(k_t(\mathbf{x},\mathbf{y})-{h_t}(\mathbf{x},\mathbf{y})\right)\right]\\&+\frac{1}{\sqrt{t}}\left[(y_j-x_j)\left(\frac{1}{2t_1}-\frac{1}{2t}\right){h_t}(\mathbf{x},\mathbf{y})\right]+\frac{1}{\sqrt{t}}\left[2x_j\left(\frac{1}{4t_1}-\frac{1}{2}\coth(2t)\right)k_t(\mathbf{x},\mathbf{y})\right]\\&=S_1+S_2+S_3.
\end{align*}
\textbf{Estimation for $S_1.$} \textit{Case 1.} $\dz(\mathbf{x},\mathbf{y}) > 2\rho(\mathbf{y}_0)$. We consider each summand separately. By Lemma~\ref{lem:t_t_1_estimations}~\eqref{numitem:t<t_1} we have $t \sim t_1$. By Theorem~\ref{teo:theoremGauss}~\eqref{numitem:space} with $m=0$,~\eqref{eq:formula_Dj}, and~\eqref{eq:connection_Dunkl_Dunkl_Hermite} we obtain
\begin{align*}
\frac{|x_j-y_j|}{t^{3/2}}k_t(\mathbf{x},\mathbf{y})& \lesssim t^{-1/2}|T_{j,\mathbf{x}}{h_{t_1}}(\mathbf{x},\mathbf{y})| \lesssim \frac{1}{t}\mathcal{G}_{t/c}(\mathbf{x},\mathbf{y})\\&\lesssim \frac{1}{t}\frac{t}{\dz^2(\mathbf{x},\mathbf{y})}\mathcal{G}_{2t/c}(\mathbf{x},\mathbf{y}) \lesssim \rho(\mathbf{y}_0)^{-2}\mathcal{G}_{2t/c}(\mathbf{x},\mathbf{y}) .
\end{align*}
The estimation for $\frac{|x_j-y_j|}{t^{3/2}}{h_t}(\mathbf{x},\mathbf{y}) \lesssim t^{-1/2}|T_{j,\mathbf{x}}{h_t}(\mathbf{x},\mathbf{y})|$ goes in the same way.
\\
\textit{Case 2.} $\dz(\mathbf{x},\mathbf{y}) \leq 2\rho(\mathbf{y}_0)$. Similarly as in the proof of Proposition~\ref{propo:H_t_H_t_1_difference}, we express the difference as a sum of two terms
\begin{align*}
&\frac{1}{\sqrt{t}}\left[\frac{y_j-x_j}{2t_1}\left(k_t(\mathbf{x},\mathbf{y})-{h_t}(\mathbf{x},\mathbf{y})\right)\right]=\frac{1}{\sqrt{t}}\left[\frac{y_j-x_j}{2t_1}\left({h_{t_1}}(\mathbf{x},\mathbf{y})-{h_t}(\mathbf{x},\mathbf{y})\right)\right]\\&+\frac{1}{\sqrt{t}}\left[\frac{y_j-x_j}{2t_1}\left(k_t(\mathbf{x},\mathbf{y})-{h_{t_1}}(\mathbf{x},\mathbf{y})\right)\right]=S_{1,1}+S_{1,2}.
\end{align*}
Since $|x_j-y_j| \leq |x_j|+|y_j| \lesssim \rho(\mathbf{y}_0)^{-1}$, by Lemma~\ref{lem:H_t_H_t_1_difference}~\eqref{numitem:H_t_H_t1_difference}, we have
\begin{align*}
|S_{1,1}| \lesssim \rho(\mathbf{y}_0)^{-1}t^{1/2}\mathcal{G}_{t/c}(\mathbf{x},\mathbf{y}) \lesssim \rho(\mathbf{y}_0)^{-2}\mathcal{G}_{2t/c}(\mathbf{x},\mathbf{y}).
\end{align*}
In order to estimate $S_{1,2}$, we use Lemma~\ref{lem:H_t_H_t_1_difference}~\eqref{numitem:Dunkl_Dunkl_Hermite_difference} and the fact that in this case $\|\mathbf{x}\|^2+\|\mathbf{y}\|^2 \lesssim \rho(\mathbf{y}_0)^{-2}$. By Theorem~\ref{teo:theoremGauss}~\eqref{numitem:space} and~\eqref{eq:formula_Dj}, we have
\begin{align*}
&|S_{1,2}| \lesssim \frac{|x_j-y_j|}{t^{3/2}}t\rho(\mathbf{y}_0)^{-2}{h_{t_1}}(\mathbf{x},\mathbf{y}) \lesssim \rho(\mathbf{y}_0)^{-2}t^{1/2}|T_{j,\mathbf{x}}{h_{t_1}}(\mathbf{x},\mathbf{y})|\lesssim \rho(\mathbf{y}_0)^{-2}\mathcal{G}_{t/c}(\mathbf{x},\mathbf{y}).
\end{align*}
\\
\textbf{Estimation for $S_2$.} Using Lemma~\ref{lem:t_t_1_estimations}~\eqref{numitem:t^3} and Theorem~\ref{teo:theoremGauss}~\eqref{numitem:Gauss} we have
\begin{align*}
|S_2| \lesssim t\mathcal{G}_{t/c}(\mathbf{x},\mathbf{y}) \lesssim \rho(\mathbf{y}_0)^{-2}\mathcal{G}_{t/c}(\mathbf{x},\mathbf{y}).
\end{align*}
\textbf{Estimation for $S_3$.} \textit{Case 1.} $\|\mathbf{x}\| \leq 4\rho(\mathbf{y}_0)^{-1}$. Then, by Lemma~\ref{lem:t_t_1_estimations}~\eqref{numitem:coth_t}, we have
\begin{align*}
|S_3| \lesssim \rho(\mathbf{y}_0)^{-1}t^{1/2}\mathcal{G}_{t/c}(\mathbf{x},\mathbf{y}) \lesssim \rho(\mathbf{y}_0)^{-2}\mathcal{G}_{t/c}(\mathbf{x},\mathbf{y}).
\end{align*}
\textit{Case 2.} $\|\mathbf{x}\| > 4\rho(\mathbf{y}_0)^{-1}$. Then $\dz(\mathbf{x},\mathbf{y}) \geq \rho(\mathbf{y}_0)^{-1}$ and $\|\mathbf{x}\| \lesssim \dz(\mathbf{x},\mathbf{y})$. Thus, by Lemma~\ref{lem:t_t_1_estimations}~\eqref{numitem:coth_t}, we have
\begin{align*}
|S_3| \lesssim \|\mathbf{x}\|t^{1/2}\mathcal{G}_{t/c}(\mathbf{x},\mathbf{y}) \lesssim \|\mathbf{x}\|\frac{t}{\dz(\mathbf{x},\mathbf{y})}\mathcal{G}_{2t/c}(\mathbf{x},\mathbf{y}) \lesssim t\mathcal{G}_{2t/c}(\mathbf{x},\mathbf{y}) \lesssim \rho(\mathbf{y}_0)^{-2}\mathcal{G}_{2t/c}(\mathbf{x},\mathbf{y}).
\end{align*}
\end{proof}

\begin{lemma}\label{lem:R_j_tail}
Let $A>1$ and $r=\min\{\rho(\mathbf{y}),\rho(\mathbf{y}_0)\}$. There exists a constant $C>0$ such that for any $\mathbf{y}_0 \in \mathbb{R}^N$ and for any $\mathbf{y} \in B(\mathbf{y}_0,A\rho(\mathbf{y}_0))$ we have
\begin{equation}
\int_{\mathbb{R}^N}\int_{r^2}^{\infty}\frac{1}{\sqrt{t}}\left|T_{j,\mathbf{x}} k_t(\mathbf{x},\mathbf{y})\right| \,dt \,dw(\mathbf{x}) \leq C.
\end{equation}
\end{lemma}

\begin{proof}
Similarly as in the proof of Lemma~\ref{lem:atom_in_H1_Dunkl_Hermite}, we split the inner integral into integral over $[r^2,1)$ and $[1,\infty)$. For the first one, we have $1+\frac{1}{t} \lesssim r^{-2}$ and $r^{-2} \lesssim \|\mathbf{x}\|^2+\|\mathbf{y}\|^2+1$, so by Lemma~\ref{lem:estimation_T_jH_t} and Lemma~\ref{lem:t_t_1_estimations}~\eqref{numitem:coth_t} it suffices to establish
\begin{align*}
\int_{r^2}^{1}r^{-2}\exp(-ctr^{-2})\int_{\mathbb{R}^N}\mathcal{G}_{t_1/c}(\mathbf{x},\mathbf{y})\,dw(\mathbf{x})\,dt \lesssim 1.
\end{align*}
The inner integral is bounded by constant independent of $t_1$, $\mathbf{y}$, and $\mathbf{y}_0$. It remains to note that $\int_{r^2}^{1}r^{-2}\exp(-ctr^{-2})\,dt \leq \int_{r^2}^{\infty}r^{-2}\exp(-ctr^{-2})\,dt$ and the last integral is finite and independent of $r$. For the integral over $[1,\infty)$, by Lemma~\ref{lem:t_t_1_estimations}~\eqref{numitem:coth>C}, we have
\begin{align*}
&\exp\left(\frac{1}{2}\left(\frac{1}{4t_1}-\frac{1}{2}\coth(2t)\right)\left(\|\mathbf{x}\|^2+\|\mathbf{y}\|^2\right)\right) \leq \exp\left(-\frac{1}{8}\left(\|\mathbf{x}\|^2+\|\mathbf{y}\|^2\right)\right).
\end{align*}
On the other hand, by~\eqref{eq:measure_of_balls} and Lemma~\ref{lem:t_t_1_estimations}~\eqref{numitem:ext-t}, we see that
\begin{align*}
\mathcal{G}_{t_1/c}(\mathbf{x},\mathbf{y}) \lesssim \frac{1}{B(\mathbf{x},\sqrt{t_1})} \lesssim t_1^{-N/2}\frac{1}{w(B(\mathbf{x},1))} \lesssim \exp(-Nt)\frac{1}{w(B(\mathbf{x},1))}.
\end{align*}
It is easy to check that the two inequalities above and Lemma~\ref{lem:estimation_T_jH_t} imply the claim.
\end{proof}

Lemma~\ref{lem:Riesz_H_t_H_t1} and Lemma~\ref{lem:R_j_tail} imply the following corollary.
\begin{corollary}\label{coro:Riesz_difference}
Let $A>1$. There exists a constant $C>0$ such that for any $\mathbf{y}_0 \in \mathbb{R}^N$ and for any $f \in L^{1}(B(\mathbf{y}_0,A\rho(\mathbf{y}_0)),dw)$, and for $j=1,2,\ldots,N$ we have
\begin{equation}
\left\|\widetilde{R}_{j}f-R_{j}^{\rho(\mathbf{y}_0)}f\right\|_{L^1(\mathbb{R}^N,dw)} \leq C\|f\|_{L^1(\mathbb{R}^N,dw)}.
\end{equation}
\end{corollary}

\begin{lemma}\label{lem:atom_in_Riesz}
For any constant $A>1$ there exists a constant $C=C(A)>0$ such that for $j=1,2,\ldots,N$ and for any function $a$ such that $\supp a \subset B(\mathbf{p},r)$ for some $\mathbf{p} \in \mathbb{R}^N$, $\sup_{\mathbf{x} \in \mathbb{R}^N}|a(\mathbf{x})| \leq w(B(\mathbf{p},r))^{-1}$ and $\rho(\mathbf{p}) \leq r \leq A\rho(\mathbf{p})$ we have
\begin{equation}
\|\widetilde{R}_ja\|_{L^1(\mathbb{R}^N,dw)} \leq C.
\end{equation}
\end{lemma}

\begin{proof}
Thanks to Corollary~\ref{coro:Riesz_difference}, we have
\begin{align*}
\|\widetilde{R}_ja\|_{L^1(\mathbb{R}^N,dw)} &\leq \|\widetilde{R}_ja-{R^{\rho(\mathbf{p})}_{j}}a\|_{L^1(\mathbb{R}^N,dw)}+\|R_j^{\rho(\mathbf{p})}a\|_{L^1(\mathbb{R}^N,dw)} \\&\lesssim \|a\|_{L^1(\mathbb{R}^N,dw)}+\|R_j^{\rho(\mathbf{p})}a\|_{L^1(\mathbb{R}^N,dw)}.
\end{align*}
By Proposition~\ref{propo:local_Riesz}, $\|R_j^{\rho(\mathbf{p})}a\|_{L^1(\mathbb{R}^N,dw)} \lesssim \|a\|_{H^{1}_{{\rm{loc}},\rho(\mathbf{p})}}$. Recall that $\rho(\mathbf{p}) \leq r$, hence, by Lemma~\ref{lem:atom_in_H1}, $\|a\|_{H^1_{{\rm loc},\rho(\mathbf{p})}} \lesssim 1$, which completes the proof.
\end{proof}

\begin{proof}[Proof of Theorem~\ref{teo:Riesz_Hermite}]
The Riesz transform $\widetilde{R}_j$ for $f \in L^{1}(\mathbb{R}^N,dw)$ is defined in the sense of distribution (see~\cite[Section 8]{conjugate}). Therefore, for $f=\sum_{j=1}^{\infty}c_ja_j$ we have $\widetilde{R}_jf=\sum_{j=1}^{\infty}c_j\widetilde{R}_ja_j$ in the sense of distributions. In order to prove the second inequality in~\eqref{eq:Riesz}, thanks to Theorem \ref{teo:atomic_Dunkl_Hermite},  it is enough to prove that there is a constant $C>0$ such that
\begin{equation}
\label{eq:atom_constant}
 \|\widetilde{R}_ja\|_{L^1(\mathbb{R}^N,dw)} \leq C
\end{equation}
for any atom $a$ of $\mathcal{H}^{1,{{\rm{at}}}}$.  Fix an atom $a(\cdot)$ for $\mathcal{H}^{1,\rm{at}}$. Let $B(\mathbf{x},r)$ be as in Definition~\ref{def:atomic_Dunkl_Hermite}. We have two cases: $r<\rho(\mathbf{x})$ or $\rho(\mathbf{x}) \leq r \leq A\rho(\mathbf{x})$. In the first case,~\eqref{eq:atom_constant} is a consequence of Corollary~\ref{coro:Riesz_difference}, Theorem~\ref{teo:maximal_atom}, and Proposition~\ref{propo:local_Riesz}. In the second case,~\eqref{eq:atom_constant} follows by Lemma~\ref{lem:atom_in_Riesz}.

For the proof of the first inequality in~\eqref{eq:Riesz}, suppose that $f,\, \widetilde{R}_jf \in L^1(\mathbb{R}^N,dw)$ for $j=1,2,\ldots,N$. Let $\{\psi_m\}$ be the partition of unity from Lemma~\ref{lem:partition_of_unity_invariant}. Then $f=\sum_{m=1}^{\infty}f\psi_m$. Similarly as in the proof of Theorem~\ref{teo:atomic_Dunkl_Hermite}, by Proposition~\ref{propo:local_Riesz}, it is enough to show that
\begin{align*}
\sum_{m=1}^{\infty}\|R_{j}^{\rho(\mathbf{y}_m)}(f\psi_m)\|_{L^1(\mathbb{R}^N,dw)} \lesssim \|f\|_{L^1(\mathbb{R}^N,dw)}+\|\widetilde{R}_jf\|_{L^1(\mathbb{R}^N,dw)}
\end{align*}
for $j=1,2,\ldots,N$. We have
\begin{align*}
&\sum_{m=1}^{\infty}\|R_{j}^{\rho(\mathbf{y}_m)}(f\psi_m)(\cdot)\|_{L^1(\mathbb{R}^N,dw)} \leq \sum_{m=1}^{\infty}\|R_{j}^{\rho(\mathbf{y}_m)}(f\psi_m)(\cdot)-\widetilde{R}_j(f\psi_m)(\cdot)\|_{L^1(\mathbb{R}^N,dw)}\\&+\sum_{m=1}^{\infty}\|\widetilde{R}_j(f\psi_m)(\cdot)-\psi_m(\cdot)\widetilde{R}_j f(\cdot)\|_{L^1(\mathbb{R}^N,dw)}+\sum_{m=1}^{\infty}\|\psi_m(\cdot)\widetilde{R}_j f(\cdot)\|_{L^1(\mathbb{R}^N,dw)}\\&=S_1+S_2+S_3.
\end{align*}
By Corollary~\ref{coro:Riesz_difference}, $S_1 \lesssim \|f\|_{L^1(\mathbb{R}^N,dw)}$. Obviously, $S_3 \lesssim \|\widetilde{R}_jf\|_{L^1(\mathbb{R}^N,dw)}$. Thus, it remains to show that $S_2 \lesssim \|f\|_{L^1(\mathbb{R}^N,dw)}$. Fix $\mathbf{y} \in \mathbb{R}^N$. It suffices to show that
\begin{align*}
&\sum_{m=1}^{\infty} \int_{\mathbb{R}^N}\int_{0}^{\infty} \left|\frac{1}{\sqrt{t}}T_{j,\mathbf{x}} k_t(\mathbf{x},\mathbf{y})\left[\psi_m(\mathbf{x})-\psi_m(\mathbf{y}) \right]\right|\,dt\,dw(\mathbf{x}) \leq C,
\end{align*}
with $C>0$ which does not depend on $\mathbf{y}$. We have
\begin{align*}
&\sum_{m=1}^{\infty} \int_{\mathbb{R}^N}\int_{0}^{\infty} \left|\frac{1}{\sqrt{t}}T_{j,\mathbf{x}} k_t(\mathbf{x},\mathbf{y})\left[\psi_m(\mathbf{x})-\psi_m(\mathbf{y}) \right]\right|\,dt\,dw(\mathbf{x}) \\&\leq \sum_{m=1}^{\infty} \int_{\mathbb{R}^N}\int_{0}^{\rho(\mathbf{y})^2} \left|\frac{1}{\sqrt{t}}T_{j,\mathbf{x}} k_t(\mathbf{x},\mathbf{y})\left[\psi_m(\mathbf{x})-\psi_m(\mathbf{y}) \right]\right|\,dt\,dw(\mathbf{x})\\&+\sum_{m=1}^{\infty} \int_{\mathbb{R}^N}\int_{\rho(\mathbf{y})^2}^{\infty} \left|\frac{1}{\sqrt{t}}T_{j,\mathbf{x}} k_t(\mathbf{x},\mathbf{y})\right|[\psi_m(\mathbf{x})+\psi_m(\mathbf y)]\,dt\,dw(\mathbf{x})\\
&=I_1+I_2.
\end{align*}
Clearly, by Lemma~\ref{lem:R_j_tail}, $I_2\leq C$. Next, thanks to  Lemma~\ref{lem:estimation_T_jH_t}, Lemma~\ref{lem:psi_m_cancellation}, and Lemma~\ref{lem:t_t_1_estimations} we see that there is $C_1>0$ such that
\begin{align*}
\left|\frac{1}{\sqrt{t}}T_{j,\mathbf{x}} k_t(\mathbf{x},\mathbf{y})\left[\psi_m(\mathbf{x})-\psi_m(\mathbf{y}) \right]\right| \leq C_1\frac{\rho^{-1}(\mathbf{y})}{\sqrt{t}}\mathcal{G}_{t/c}(\mathbf{x},\mathbf{y})
\end{align*}
for $\mathbf{x} \in B(\mathbf{y}_m,3\rho(\mathbf{y}_m))$ or for $\mathbf{y} \in B(\mathbf{y}_m,3\rho(\mathbf{y}_m))$. Moreover, $|\psi_m(\mathbf{y})-\psi_m(\mathbf{x})|=0$ if $\mathbf{x},\mathbf{y} \in B(\mathbf{y}_m,3\rho(\mathbf{y}_m))^{c}$. {Note that there is at most $M$ indices $m$ such that $\mathbf{y} \in B(\mathbf{y}_m,3\rho(\mathbf{y}_m))$.} {Therefore, by the finite covering property of the balls $B(\mathbf{y}_m,3\rho(\mathbf{y}_m))$,
\begin{align*}
I_1\leq &C_1\sum_{m=1}^{\infty} \int_{B(\mathbf{y}_m,3\rho(\mathbf{y}_m))}\int_{0}^{\rho^2(\mathbf{y})} \frac{\rho^{-1}(\mathbf{y})}{\sqrt{t}}\mathcal{G}_{t/c}(\mathbf{x},\mathbf{y}) \,dt\,dw(\mathbf{x})\\&+C_1M\int_{\mathbb{R}^N}\int_{0}^{\rho^2(\mathbf{y})} \frac{\rho^{-1}(\mathbf{y})}{\sqrt{t}}\mathcal{G}_{t/c}(\mathbf{x},\mathbf{y}) \,dt\,dw(\mathbf{x})\\
&\leq 2C_1M\int_{0}^{\rho^2(\mathbf{y})} \int_{\mathbb{R}^N} \frac{\rho^{-1}(\mathbf{y})}{\sqrt{t}}\mathcal{G}_{t/c}(\mathbf{x},\mathbf{y}) \,dw(\mathbf{x})\,dt
\leq C.
\end{align*} }
\end{proof}

{\bf Acknowledgment.} The author would like to thank Jacek {Dziuba\'nski} and Jean-Philippe Anker for discussions on the topics  
considered in the paper.

\end{document}